\documentclass[hidelinks,onefignum,onetabnum]{siamart250211}

\usepackage{lipsum}
\usepackage{amsfonts}
\usepackage{graphicx}
\usepackage{epstopdf}
\usepackage{algorithmic}
\ifpdf
  \DeclareGraphicsExtensions{.eps,.pdf,.png,.jpg}
\else
  \DeclareGraphicsExtensions{.eps}
\fi

\usepackage{xcolor} %
\usepackage{comment} 
\usepackage{bm}
\usepackage{mathtools} %
\mathtoolsset{centercolon} %

\newsiamremark{remark}{Remark}
\newsiamremark{hypothesis}{Hypothesis}
\crefname{hypothesis}{Hypothesis}{Hypotheses}
\newsiamthm{claim}{Claim}
\newsiamremark{fact}{Fact}
\crefname{fact}{Fact}{Facts}

\crefname{subsection}{Section}{Sections}
\Crefname{subsection}{Section}{Sections}
\crefname{section}{Section}{Sections}

\headers{Optimal transfer operators in algebraic two-level methods}{Krzysik, Southworth, Wimmer, Ali, Brannick, Kahl}

\title{Optimal transfer operators in algebraic two-level methods for nonsymmetric and indefinite problems\thanks{Submitted to the editors DATE.\funding{This research has been funded by the Los Alamos National Laboratory Advanced Simulation and Computation, and the DOE Office of Advanced Scientific Computing Research Applied Mathematics program through Contract No. 89233218CNA000001. The research was performed under the auspices of the National Nuclear Security Administration of the U.S. Department of Energy at Los Alamos National Laboratory, managed by Triad National Security, LLC under contract 89233218CNA000001. LA-UR-25-24473.
}}}
\author{Oliver A. Krzysik\thanks{Theoretical Division, Los Alamos National Laboratory, Los Alamos, NM 87545, USA (\email{okrzysik@lanl.gov})
}
\and
Ben S. Southworth\thanks{Theoretical Division, Los Alamos National Laboratory, Los Alamos, NM 87545, USA (\email{southworth@lanl.gov})
}
\and
Golo A. Wimmer\thanks{Theoretical Division, Los Alamos National Laboratory, Los Alamos, NM 87545, USA (\email{gwimmer@lanl.gov})
}
\and
Ahsan Ali\thanks{Department of Mathematics, Baylor University, Waco, TX 76706 USA
(\email{ahsan\_ali2@baylor.edu})
}
\and
James Brannick\thanks{Department of Mathematics, Penn State University, University Park, PA 16802, USA
(\email{jbrannick@gmail.com})
}
\and
Karsten Kahl\thanks{School of Mathematics and Natural Sciences, Bergische Universit\"{a}t Wuppertal, Wuppertal 42119, Germany (\email{kkhal@uni-wuppertal.de})
}}

\usepackage{amsopn}
\DeclareMathOperator{\diag}{diag}

\allowdisplaybreaks 

\newcommand{\wh}[1]{\widehat{#1}} 						\newcommand{\wt}[1]{\widetilde{#1}}
\DeclareFontFamily{U}{mathx}{}
\DeclareFontShape{U}{mathx}{m}{n}{<-> mathx10}{}
\DeclareSymbolFont{mathx}{U}{mathx}{m}{n}
\DeclareMathAccent{\wc}{0}{mathx}{"71} %
\DeclareMathAccent{\wbar}{0}{mathx}{"73} %

\renewcommand{\Re}{\textnormal{real}}
\renewcommand{\Im}{\textnormal{imag}}

\newcommand{\cc}[1]{\overline{#1}}
\newcommand{\ra}[0]{\ensuremath{\rangle}}
\newcommand{\la}[0]{\ensuremath{\langle}}

\DeclareMathOperator{\rank}{rank}
\DeclareMathOperator{\ran}{range}
\DeclareMathOperator{\nul}{null}
\renewcommand{\d}[0]{\ensuremath{\operatorname{d}\!}} %
\DeclareMathOperator*{\argmin}{arg\,min}

\begin{document}

\maketitle

\begin{abstract}
Consider an algebraic two-level method applied to the $n$-dimensional linear system $A \bm{x} = \bm{b}$ using fine-space preconditioner (i.e., ``relaxation'' or ``smoother'') $M$, with $M \approx A$, restriction and interpolation $R$ and $P$, and algebraic coarse-space operator ${A_c\coloneqq R^*AP}$. Then, what are the the best possible transfer operators $R$ and $P$ of a given dimension $n_c < n$? 
    Brannick et al. \cite{Brannick-etal-2018-optimalP} showed that when $A$ and $M$ are Hermitian positive definite (HPD), the optimal interpolation is such that its range contains the $n_c$ smallest generalized eigenvectors of the matrix pencil $(A, M)$.
    Recently, in Ali et al. \cite{Ali-etal-2024-optimalP-gen} we generalized this framework to the non-HPD setting, by considering both right (interpolation) and left (restriction) generalized eigenvectors of $(A, M)$ and defining corresponding nonsymmetric transfer operators $\{R_\#,P_\#\}$. Tight convergence bounds for $\{R_\#,P_\#\}$ are derived in spectral radius, as well as a proof of pseudo-optimality. Note, $\{R_\#,P_\#\}$ are typically complex valued, which is not practical for real-valued problems.
    Here we build on \cite{Ali-etal-2024-optimalP-gen}, first characterizing all inner products in which the coarse-space correction defined by $\{R_\#,P_\#\}$ is orthogonal. 
    We then develop tight two-level convergence bounds in these norms, and prove that the underlying transfer operators $\{R_\#,P_\#\}$ are genuinely optimal.
    As a special case, our theory both recovers and extends the HPD results from \cite{Brannick-etal-2018-optimalP}.
    Finally, we show how to construct optimal, \emph{real-valued} transfer operators in the case of that $A$ and $M$ are real valued, but are not HPD.
    Numerical examples arising from discretized advection and wave-equation problems are used to verify and illustrate the theory.
\end{abstract}



\section{Introduction}
We consider iterative, algebraic two-level methods for linear systems $A \bm{x} = \bm{b}$, for an invertible matrix ${A \in \mathcal{F}^{n \times n}}$, a known vector $\bm{b} \in \mathcal{F}^n$, and a field $\mathcal{F}$ given by $\mathbb{R}$ or $\mathbb{C}$.
Algebraic two-level methods are built on two fundamental components: 
(i) a fine-space preconditioner (often called ``relaxation'' or ``smoothing'' in the context of multigrid) with iteration given by
\begin{align} \label{eq:relax}
    \bm{x}_{k} \gets \bm{x}_k + M^{-1} ( \bm{b} - A \bm{x}_k ),
\end{align}
where $M \approx A$ is a preconditioner whose inversion should be inexpensive relative to that of $A$, and (ii) a coarse-space correction,
\begin{align} \label{eq:cgc}
    \bm{x}_k \gets \bm{x}_k + P (R^* A P)^{-1} R^* ( \bm{b} - A \bm{x}_k ), 
\end{align}
where $P \in \mathcal{F}^{n \times n_c}$ interpolates corrections from an $n_c$-dimensional coarse subspace, and $R \in \mathcal{F}^{n \times n_c}$ restricts residuals from the fine space to the coarse space.
Here $k$ is an iteration index.
Letting $\bm{e}_k = \bm{x} - \bm{x}_k$ denote the error in the approximation $\bm{x}_k$, then error propagation of \eqref{eq:relax} takes the form $\bm{e}_k \gets (I - M^{-1}A) \bm{e}_k$ and error propagation of \eqref{eq:cgc} is given by
$\bm{e}_k \gets (I - \Pi(P, R))\bm{e}_k$, where ${\Pi(P, R) \coloneqq P (R^* A P)^{-1} R^* A}$ is a projection onto $\ran(P)$, such that the step \eqref{eq:cgc} eliminates error in $\bm{x}_k$ that resides in the subspace ${\ran(P) \subset \mathcal{F}^n}$. An efficient two-level method necessitates that the two components \eqref{eq:relax} and \eqref{eq:cgc} be \emph{complementary} to one another.

There are many forms of algebraic two-level methods, including algebraic domain decomposition (DD) \cite{Al-Daas-2021-DD,al2025robust}, algebraic multigrid (AMG) \cite{BRANDT-1986-AMG,Ruge-Stuben-1987,Falgout-Vassilevski-2004-AMG-generalizing,Falgout-etal-2005-two-grid-conv}, element-based AMG \cite{Brezina-etal-2001-AMGe,Chartier-etal-2003-spectral-AMG}, multiscale finite elements \cite{yang2019two,aarnes2002multiscale,galvis2010domain}, algebraic multilevel iterations \cite{axelsson1989algebraic,axelsson1990algebraic,axelsson2003survey}, etc; see also \cite{Notay-2005} for a comparison of some of these methods.
When $A$ is Hermitian positive definite (HPD), $A$ naturally induces an energy norm, and thus a basis in which to understand error reduction. 
In this setting, one would typically use $R = P$, such that $\Pi(P, P)$ is an $A$-orthogonal projection onto $\ran(P)$. Convergence of two-level methods for HPD $A$ is well understood, including tight two-level convergence bounds, and \emph{optimal} interpolation operators given a prescribed fine-space preconditioner, e.g. \cite{Brannick-etal-2018-optimalP,Falgout-etal-2005-two-grid-conv,Falgout-Vassilevski-2004-AMG-generalizing,vassilevski2010lecture}. The optimal interpolation formula has led to various local approximations for practical algorithms, e.g. \cite{Brannick-etal-2018-optimalP,doi:10.1137/100796376,galvis2010domain,doi:10.1137/100790112}.

For non-HPD $A$ things are more complicated. There is no natural energy norm to work in, so it is typically unclear what basis the error should be represented in. As a result, the coarse-space correction is typically not orthogonal in a known inner product, and if it is not, then it must necessarily increase error in \emph{all} norms \cite{Southworth-Manteuffel-2024-AMG}, putting stringent requirements on the effectiveness of $M$. Although there are effective methods for certain nonsymmetric problems, e.g. \cite{BAMG_Wilson,manteuffel2017root,manteuffel2018nonsymmetric,manteuffel2019nonsymmetric,Wiesner:2014cy,Sala:2008cv,Notay:2000vy,clair,dargaville2024air}, there remain problems for which no effective methods are known. 
 Moreover, though various results have been obtained for non-HPD systems, there does not exist a cohesive, unifying framework for designing effective algebraic two-level methods.

Another aspect of nonsymmetric convergence theory that has not been completely addressed is how to  incorporate non-point-wise smoothers in the two-level solver. In \cite{Manteuffel-Southworth-2019-NS-AMG} general two-level convergence and approximation properties are considered for nonsymmetric problems with Richardson preconditioning in terms of left and right singular vectors. Block F-relaxation is considered in \cite{manteuffel2018nonsymmetric,manteuffel2019nonsymmetric} due to its special relation to ``ideal'' transfer operators. Recent work has also considered fine-level Schwarz preconditioners in algebraic multilevel methods, e.g. \cite{al2025robust}. Fine-space preconditioning is incorporated to some extent in the spectral nonsymmetric theory developed in \cite{Notay:2010em}, but the question of how, given a fine-space preconditioner, to construct complementary transfer operators remains open for general non-HPD problems. An even more fundamental issue than complementarity for non-HPD problems is that fine-space preconditioning in the nonsymmetric setting is often not even guaranteed to converge; for example, typical point-wise fine-space preconditioners such as Jacobi or Gauss-Seidel tend not to have convergence guarantees. As such, it is difficult to contextualize the role and effects of $M$, and hence understand what exactly it means for the coarse space to be complementary to it. Such challenges motivate the use of more complicated and advanced fine-space preconditioners, but the design of effective two-level methods still requires us to understand their relation to transfer operators. For example, in the symmetric positive-definite (SPD) setting two-sided convergence bounds are used in \cite{zikatanov2008two} to \emph{prove} that a two-level method with point-wise smoother for a variational problem in $H_0($curl$)$ cannot have optimal convergence rate. We seek to extend such analysis capabilities to nonsymmetric and indefinite problems.

This paper addresses the questions \emph{given matrix $A$ and fine-space preconditioner $M$, (i) is there a convergent algebraic two-level method of coarse size $n_c< n$, and (ii) what are the optimal interpolation and restriction operators?} The optimal transfer operators we derive are not immediately practical, but the results offer insight on the design of nonsymmetric two-level methods, particularly for more complicated fine-space preconditioning schemes. We build on recent work by Ali et al. \cite{Ali-etal-2024-optimalP-gen}, where transfer operators $\{R_\#,P_\#\}$ are proposed based on left and right generalized eigenvectors, respectively. For two-level methods based on $\{R_\#,P_\#\}$, tight convergence bounds are derived with respect to the spectral radius of error propagation, and a proof of pseudo-optimality of $\{R_\#,P_\#\}$ is provided, in the sense of $\{R_\#,P_\#\}$ being optimal over a certain subspace of restriction and interpolation operators. As it turns out, we prove in this paper that the transfer operators derived in \cite{Ali-etal-2024-optimalP-gen} are \emph{genuinely optimal} with respect to a full class of norms. Here we strengthen the results from \cite{Ali-etal-2024-optimalP-gen}, first characterizing all inner products in which the resulting coarse-space correction defined by $\{R_\#,P_\#\}$ is orthogonal. 
We then develop tight two-level convergence bounds in a subset of these norms, and prove that the underlying (pseudo-)optimal transfer operators from \cite{Ali-etal-2024-optimalP-gen} are genuinely optimal. This leads to a set of necessary and sufficient conditions for a convergent two-level method, given $A$ and $M$. Finally, noting that coarse-grid correction is invariant under change of basis, e.g. \cite{Brannick-etal-2018-optimalP,Manteuffel-Southworth-2019-NS-AMG}, we develop a coarse-space change of basis that yields optimal \emph{real-valued} transfer operators in the case of $A,M\in\mathbb{R}^{n\times n}$, with identical convergence properties as the optimal complex transfer operators.

The remainder of this paper is organized as follows.
\Cref{sec:prelims} presents preliminaries, including notation and assumptions, and a review of previous optimal transfer operator results in the HPD \cite{Brannick-etal-2018-optimalP} and non-HPD \cite{Ali-etal-2024-optimalP-gen} settings.
\Cref{sec:norm-conv-opt} develops a class of norms in which to consider convergence of algebraic two-level methods based on optimal transfer operators. It also establishes the optimality of these operators in this class of norms.
\Cref{sec:real-valued} develops real-valued optimal transfer operators when the pencil $(A, M)$ is real valued, and presents theoretical convergence results for them that are equivalent to those from \cref{sec:norm-conv-opt} for their complex-valued counterparts. 
\Cref{sec:num} presents numerical results to complement the theoretical results, considering a hyperbolic advection-reaction equation and wave equation. We demonstrate good predictive accuracy of convergence bounds (theoretically posed in terms of spectral radius or nonstandard inner products) against observed $\ell^2$-error and $\ell^2$-residual convergence, and demonstrate various practical insights in developing two-level algebraic methods. Conclusions and future work are discussed in \cref{sec:con}.

\section{Preliminaries}
\label{sec:prelims}

\subsection{Notation}
\label{sec:notation}
We consider real or complex square matrix pencils $(A,M)$ for $A,M\in\mathbb{C}^{n\times n}$ or $\mathbb{R}^{n\times n}$. Throughout the paper we assume that $M^{-1}A$ and $M^{-*}A^*$ are diagonalizable, which provides sufficient conditions for the left and right generalized eigenvectors to be well defined and invertible. A superscript ``$*$'' denotes the complex-conjugate transpose of a matrix and a superscript ``$\top$'' denotes the regular transpose of a matrix. 
We assume a splitting of the matrix dimension $n = n_c+n_f$, for coarse-space dimension $n_c \in [1, n]$, and fine-space dimension $n_f  \coloneqq  n - n_c \in [0, n-1]$. We let $R, P\in\mathbb{C}^{n\times n_c}$ (or $\mathbb{R}^{n\times n_c}$), wherein the restriction operator is applied via $R^*$. A subscript on a matrix of the form $[a:b, a:b]$ denotes the submatrix given by rows and columns indexed $a,...,b$; a subscript on a matrix of the form $[a:b]$ denotes all columns $a,...,b$. We order matrix degrees of freedom (DOFs) with C-points first and F-points last. Although this convention is non-standard with respect to AMG literature, it is more natural in the present setting where the first $n_c$ eigenvectors in a particular basis induce a coarse space. When matrices are ordered based on C- and F-DOFs we use subscripts $cc, cf, fc, ff$ to denote the associated submatrices.
If $\{ \lambda_j \}_{j = 1}^n$ are the generalized eigenvalues of the matrix pencil $(A, M)$, then they are ordered according to $|1 - \lambda_1| \geq |1 - \lambda_2| \geq \cdots \geq |1 - \lambda_n| \geq 0$.
Let $\Vert \cdot \Vert  \coloneqq  \la \cdot, \cdot \ra^{1/2}$ denote the regular $\ell^2$-norm, and $\Vert \cdot \Vert_{{\cal S}}  \coloneqq  \la {\cal S} \cdot, \cdot \ra^{1/2}$ denote the induced ${\cal S}$-norm for any HPD matrix ${\cal S}$ (or SPD for real-valued operators).
An operator ${\cal A}$ is orthogonal in the ${\cal S}$-inner product, or simply ``${\cal A}$ is ${\cal S}$-orthogonal,'' if 
$\la {\cal A} \cdot, \cdot \ra_{{\cal S}} = \la \cdot,  {\cal A} \cdot \ra_{{\cal S}}$.
Last, to avoid special cases of singular coarse spaces, we assume $R^*AP$ is invertible. We point out that this has no effect on optimality though, as a singular coarse space, even if addressed properly with a pseudoinverse, would necessarily make $\ran(\Pi)$ of smaller dimension $\hat{n}<n_c = \mathrm{dim}(\ran(P))$ for coarse-space projection $\Pi\coloneqq P(R^*AP)^{-1}R^*A$. As a result, convergence cannot be better than obtained with full rank $R^*AP$ and the corresponding projection of dimension $n_c$.

\subsection{Review of existing results in the non-HPD setting}
\label{sec:review}

The error propagation operator for a two-level method with fine-space preconditioner $M \approx A$ using ${\nu_1, \nu_2\in \mathbb{Z}_{\geq 0}}$ pre- and post-fine-space-preconditioner iterations, respectively, is given by
\begin{equation} \label{eq:two-grid}
    \begin{aligned}
    E_{\rm TG}^{\nu_1, \nu_2}(P, R) 
    &\coloneqq 
    (I - M^{-1}A)^{\nu_2} [I - \Pi(P, R)] (I - M^{-1}A)^{\nu_1},
    \\
    \Pi(P, R) 
    &:=
    P(R^*AP)^{-1}R^*A.
    \end{aligned}
\end{equation}
We also denote one iteration of this method as a two-level V$(\nu_1,\nu_2)$ cycle.
In the context of HPD $A$ and $M$, where one takes $R = P$, Brannick et al. \cite{Brannick-etal-2018-optimalP} showed that any $P$ whose range spans the $n_c$ smallest eigenvectors of the generalized eigenvalue problem characterized by the pencil $(A, M)$ minimizes the A-norm of two-level V(0, 1) error propagation.
These results were extended to the nonsymmetric setting by Ali et al. \cite{Ali-etal-2024-optimalP-gen}, which we now review. We begin with a result on a matrix-induced orthogonality arising from a generalized eigenvalue problem that is fundamental to the derivation of the optimal operators.

\begin{lemma}[Lemma 4.1 in \cite{Ali-etal-2024-optimalP-gen}]
\label{lem:orth-complex}
Let $A,M \in\mathbb{C}^{n\times n}$ be such that $M$ is invertible and $M^{-1}A$ is diagonalizable. Consider the left and right generalized eigenvectors, $V_{l},V_{r}\in\mathbb{C}^{n\times n}$, respectively, of the matrix pencil $(A,M)$, defined such that
    \begin{subequations}\label{eq:gep-complex}
    \begin{align}
        AV_{r} & = MV_{r}\Lambda, \\
       \ V_{l}^*A & = \Lambda V_{l}^*M,\label{eq:gep-complex-left}
    \end{align}
    \end{subequations}
    where $\Lambda \in\mathbb{C}^{n\times n}$ is a diagonal matrix of eigenvalues. Then $V_{l}$ and $V_{r}$ induce a matrix-based orthogonality, satisfying
    \begin{subequations}
    \begin{align}
       V_{l}^*AV_{r} &= D_{a}, \\
       V_{l}^*MV_{r} &= D_{m},
    \end{align}
    \end{subequations}
    for diagonal matrices $D_{a},D_{m}\in\mathbb{C}^{n\times n}$.
    
\end{lemma}

The paper \cite{Ali-etal-2024-optimalP-gen} uses the above lemma to study the following results regarding interpolation and restriction operators $P$ and $R$ with respect to the convergence of a two-level V$(1,0)$ cycle, with a fine-space preconditioner $M$. 
First, the spectral radius of $E_{\rm TG}^{1, 0}$ is derived for transfer operators with columns defined by some set of left and right generalized eigenvectors of $(A,M)$. 

\begin{theorem}[Theorem 5.1 in \cite{Ali-etal-2024-optimalP-gen}]\label{th:optimalPR_nonsymmetric-complex}
Let $A, M\in\mathbb{C}^{n\times n}$ be non-singular, and such that $M^{-1}A$ is diagonalizable. 
Consider the left and right generalized eigenvectors, $V_{l}=\left[\bm{v}_{l,i}\right]_{i=1}^{n}$ and $V_{r}=\left[\bm{v}_{r,i}\right]_{i=1}^{n}$, of the matrix pencil $(A,M)$, respectively, as defined in \eqref{eq:gep-complex} where the corresponding eigenvalues are $\left\{\lambda_{i}\right\}_{i=1}^{n}$.  
Given a subset $\mathcal{I}\subseteq \{1,2,\dots,n\}$ with cardinality $|\mathcal{I}|=n_{c}$, define spaces of interpolation and restriction operators such that their columns span some subset of right and left generalized eigenvectors, respectively:
\begin{subequations}\label{eq:transfer-evec}
\begin{align}
\mathfrak{P} &\coloneqq \{ P\in\mathbb{C}^{n\times n_c} \textnormal{ : }  \operatorname{range}(P) =  \operatorname{range}\left(
\begin{bmatrix} \bm{{v}}_{r,\ell} \textnormal{ for }\ell \in \mathcal{I}  \end{bmatrix}\right),\label{eq:Pevec}\\
\mathfrak{R} &\coloneqq \{ R\in\mathbb{C}^{n\times n_c} \textnormal{ : }  \operatorname{range}(R) =  \operatorname{range}\left(
\begin{bmatrix} \bm{{v}}_{l,\ell} \textnormal{ for }\ell \in \mathcal{I}  \end{bmatrix}\right).\label{eq:Revec}
\end{align}
\end{subequations}
Then, the spectral radius of the two-level error propagation operator \eqref{eq:two-grid} for ${P\in\mathfrak{P}},$ $R\in\mathfrak{R}$ is given by
\begin{equation}\label{identity_nonsymmetric}
    \rho\big(E_{\rm TG}^{1, 0}(P,R)\big)
    = 
    \begin{cases} \max_{\ell\notin \mathcal{I}}|1-\lambda_{\ell}|, & n_c<n, \\ 0, & n_c=n. \end{cases}
\end{equation}
\end{theorem}

This result is then used show pseudo-optimality of certain transfer operators $P_{\sharp}$ and $R_{\sharp}$ from \eqref{eq:transfer-evec}. 
We say pseudo-optimal because optimality is only proven over ${P\in\mathfrak{P}}$, and ${R\in\mathfrak{R}}$, as opposed to \emph{all} $P,R\in\mathbb{C}^{n\times n_c}$. 
However, in the next section we ultimately prove genuine optimality of these transfer operators, so, to keep the language concise, herein we refer to $P_{\sharp}$ and $R_{\sharp}$ as optimal rather than pseudo-optimal.

\begin{corollary}[Corollary 5.2 in \cite{Ali-etal-2024-optimalP-gen}] \label{cor:optimality_transfer-complex}
    Let $A, M\in\mathbb{C}^{n\times n}$ be non-singular, and such that $M^{-1}A$ is diagonalizable. 
Assume the corresponding generalized eigenvalues $\left\{\lambda_{i}\right\}_{i=1}^{n}$ of the matrix pencil $(A,M)$ are ordered such that $|1-\lambda_{1}|\geq |1-\lambda_{2}| \geq \cdots  \geq|1-\lambda_{n}|\geq 0$. Consider interpolation and restriction operators from the spaces in \eqref{eq:transfer-evec}, i.e., $P\in\mathfrak{P},R\in\mathfrak{R}$, and define the optimal interpolation $P_{\sharp}$ and restriction $R_{\sharp}$ to satisfy
    \begin{align*}
        \operatorname{range}(P_{\sharp}) 
        &= 
        \operatorname{range}\left(
        \begin{bmatrix} \bm{{v}}_{r,1} & \bm{{v}}_{r,2} & \cdots & \bm{{v}}_{r,n_{c}} \end{bmatrix}\right),\\
        \operatorname{range}(R_{\sharp}) 
        &= 
        \operatorname{range}\left(
        \begin{bmatrix} 
        \bm{{v}}_{l,1} & \bm{{v}}_{l,2} & \cdots & \bm{{v}}_{l,n_{c}} \end{bmatrix}\right).
    \end{align*}
    Then, over the spaces in \eqref{eq:transfer-evec}, $P_\sharp$ and $R_\sharp$ minimize the spectral radius of two-level V(1, 0) error propagation, that is,
    \begin{equation}\label{eq:cor_iden}
        \min_{P\in\mathfrak{P},R\in\mathfrak{R}} \rho(E_{\rm TG}^{1, 0}(P, R)) = 
    \rho(E_{\rm TG}^{1, 0}(P_{\sharp},R_{\sharp})) 
    = 
    \begin{cases} |1-\lambda_{n_{c}+1}|, & n_c < n, \\ 0, & n_c = n. \end{cases}
    \end{equation}
\end{corollary}
\section{Convergence and optimality in norm}
\label{sec:norm-conv-opt}

In this section we strengthen the results regarding optimal transfer operators $P_\sharp$, and $R_\sharp$. 
Recall that $V_r$ is defined as the right generalized eigenvectors of matrix pencil $(A,M)$.
Throughout this section, we let $\wt{{\cal D}}$ denote a generic, invertible, block diagonal CF-split matrix of the form
\begin{align} \label{eq:wtcalD-def}
    \wt{{\cal D}} 
    = 
    \begin{bmatrix}
        {\cal U}_{cc} & 0 \\ 0 & {\cal U}_{ff}
    \end{bmatrix} \in \mathbb{C}^{n \times n},
    \quad
    \rank( \wt{{\cal D}} ) = n,
\end{align}
with upper triangular diagonal blocks ${\cal U}_{cc} \in \mathbb{C}^{n_c \times n_c}$ and ${\cal U}_{ff} \in \mathbb{C}^{n_f \times n_f}$.
We let ${\cal D}$ denote the special case that $\wt{{\cal D}}$ is a diagonal CF-split matrix of the form
\begin{align} \label{eq:calD-def}
    {\cal D} 
    = 
    \begin{bmatrix}
        \diag( u_{1}, \ldots, u_{n_c} ) & 0 \\ 0 & \diag( u_{n_c+1}, \ldots, u_{n} )
    \end{bmatrix} \in \mathbb{C}^{n \times n}, 
    \quad
    u_j \neq 0.
\end{align}

We begin in \cref{sec:norm-conv-opt:orth} by determining a class of inner products in which $\Pi(P_\sharp,R_\sharp)$ is orthogonal, specifically those induced by HPD matrix ${\cal N}=V_r^{-*} \wt{{\cal D}}^* \wt{{\cal D}} V_r^{-1}$ (see \Cref{thm:orth-characterization}). 
In \cref{sec:norm-conv-opt:two-grid} we then develop tight two-level convergence bounds for V$(\nu_1,\nu_2)$-cycles built on transfer operators $P_\sharp$ and $R_\sharp$ in the class of ${{\cal N}=(V_r^{-*} {{\cal D}}^* {{\cal D}} V_r^{-1})}$-norms (see \Cref{thm:two-grid-conv}).
Last, in \cref{sec:norm-conv-opt:optimality} we prove that for V$(\nu_1,\nu_2)$-cycles, $P_\sharp$ and $R_\sharp$ are optimal with respect to two-level convergence in the ${\cal N}=(V_r^{-*} {{\cal D}}^* {{\cal D}} V_r^{-1})$-norm over \emph{all} possible transfer operators ${R,P\in\mathbb{C}^{n\times n_c}}$ (see \Cref{thm:optimal}).

At first glance, the ${\cal N}=(V_r^{-*} \wt{{\cal D}}^* \wt{{\cal D}} V_r^{-1})$-norm may seem somewhat strange, but it is in fact quite general and a natural extension of the $A$-norm commonly used in the HPD setting.
That is, suppose $A$ and $M$ are HPD, and consider the following generalized eigenvalue problem for the pencil $(A, M)$:
\begin{align}
    A V_r = M V_r \Lambda,
    \quad
    V_r^* A V_r = D_a,
    \quad
    V_r^* M V_r = D_m,
\end{align}
with $D_a$ and $D_m$ positive diagonal matrices.
Then we have $A = V_r^{-*} D_a V_r = V_r^{-*} D_a^{*1/2} D_a^{1/2} V_r$. That is, with the choice ${\cal D} = D_a^{1/2}$, the ${\cal N}$-norm reduces to the $A$-norm.
Similarly, with the choice of ${\cal D} = D_m^{1/2}$, the ${\cal N}$-norm reduces to the $M$-norm.
Note that in the case of Richardson iteration one has $M = \omega I$, with constant $\omega \in \mathbb{R}_{+}$, so that $I = V_r^{-*} (D_m/\omega)^{1/2} (D_m/\omega)^{1/2} V_r^{-1}$, which is equal to ${\cal N}$ when ${\cal D} = (D_m/\omega)^{1/2}$; hence, the ${\cal N}$-norm includes the $\ell^2$-norm in the case of Richardson iterations.
More generally for point-wise preconditioners $M$, if we perform a symmetric diagonal scaling of $A\mapsto D^{-1/2}AD^{-1/2}$ for diagonal $D$ of $A$, Jacobi is equivalent to Richardson, and it is known that Jacobi and Gauss-Seidel are spectrally equivalent, e.g. \cite[Sec. 3]{vassilevski2010lecture}. Thus for point-wise preconditioners $M$, the ${\cal N}$-norm resembles an $\ell^2$-norm, although more complicated block or overlapping preconditioners represent more complicated structure.
\subsection{Orthogonality of $\Pi(P_{\sharp}, R_{\sharp})$}
\label{sec:norm-conv-opt:orth}

    When designing two-level methods for nonsymmetric problems it is desirable to have a coarse-space correction that is orthogonal in some meaningful inner product \cite{Manteuffel-Southworth-2019-NS-AMG,Southworth-Manteuffel-2024-AMG}. That is, if $\Pi$ is orthogonal, then it has norm one (in the inner-product-induced norm), and hence cannot increase error (in this norm). 
    Conversely, if $\Pi$ is not orthogonal in any inner product, then it necessarily increases error in all norms.
    More generally, $\Vert \Pi \Vert_{\cal N} \sim {\cal O}(1)$ independent of problem parameters for some HPD ${\cal N}$ is referred to as a ``stable coarse-space correction.''
    An orthogonal or stable coarse-space correction is important for obtaining two-level convergence, because ultimately $\Pi$ is a correction and it should not increase error. 
    That is, a non-orthogonal $\Pi$ puts potentially quite strong requirements on the fine-space preconditioner $M$, since two-level convergence necessitates that any error magnified by $\Pi$ (and thus in the range of $P$) must be rapidly attenuated by the smoother.  

    For general $R$ and $P$, \cite[Lemma 4]{Southworth-Manteuffel-2024-AMG} provides necessary and sufficient conditions for ${\cal N}$-orthogonality.
    Operators $R$ and $P$ are called \emph{compatible} if they yield a coarse-space correction $\Pi$ that is ${\cal N}$-orthogonal. Compatibility relations between $R$ and $P$ are then derived for ${\cal N}$-orthogonal projections in several meaningful inner products, such as ${\cal N} = I$, or ${\cal N} = A^* A$. Here, the situation is reversed: We begin with transfer operators $P_\sharp$ and $R_\sharp$, and ask the question is there an inner product in which these transfer operators are compatible? The following theorem characterizes all inner products in which $\Pi(P_\sharp,R_\sharp)$ is orthogonal.

    \begin{theorem} \label{thm:orth-characterization}
        Let ${\cal N}$ be HPD. The projection $\Pi(P_{\sharp}, R_{\sharp})$ is ${\cal N}$-orthogonal if and only if ${\cal N}$ can be written in the form ${\cal N} = V_r^{-*} \wt{{\cal D}}^* \wt{{\cal D}} V_r^{-1}$, where $\wt{{\cal D}}$ is any block-diagonal matrix of the form in \eqref{eq:wtcalD-def}.
    \end{theorem}

    \begin{proof}
    We begin the proof with the following lemma that provides necessary and sufficient conditions for $\Pi(P_{\sharp}, R_{\sharp})$ to be orthogonal in the ${\cal N}$-inner product.
    \begin{lemma} \label{lem:orth-nec-and-suf}
        Let ${\cal N}$ be HPD, and $V_r$ be the right generalized eigenvectors of the matrix pencil $(A, M)$.
        Then, $\Pi(P_{\sharp}, R_{\sharp})$ is orthogonal in the ${\cal N}$-inner product if and only if $V_r^* {\cal N} V_r$ is CF-block diagonal.
    \end{lemma}

    \begin{proof}
        \cite[Lemma 4]{Southworth-Manteuffel-2024-AMG} states necessary and sufficient conditions for $\Pi(P, R)  \coloneqq  P (R^* A P)^{-1} P A$ to be ${\cal N}$-orthogonal with respect to HPD matrix ${\cal N}$ are the existence of coarse-space nonsingular change-of-basis matrices $B_P$ and $B_R$ such that:\footnote{Note: \cite[Lemma 4]{Southworth-Manteuffel-2024-AMG} states ${\cal N}$ is SPD, but the result holds for complex fields and HPD ${\cal N}$ as well.}
        \begin{align} \label{eq:calM-def}
            {\cal N} P B_P = A^* R B_R.
        \end{align}
        We now plug information about $P = P_{\sharp}$ and $R = R_{\sharp}$ into \eqref{eq:calM-def}.
        Recall that $P_{\sharp}$ and $R_{\sharp}$ are defined such that their ranges are equal to those of $V_{r, 1:n_c}$ and $V_{l, 1:n_c}$ (see \cref{cor:optimality_transfer-complex}). 
        As such, we can write, $P = V_{r, 1:n_c} \wh{B}_P$ and $R = V_{l, 1:n_c} \wh{B}_R$ for arbitrary, nonsingular, coarse-space change-of-basis matrices $\wh{B}_P$ and $\wh{B}_R$.

        Next, recall from \cref{lem:orth-complex} the identity that $V_{l}^{*} A V_r = D_a$, so that $A^* = V_r^{-*} D_a^* V_l^{-1}$. Plugging into \eqref{eq:calM-def} yields
        ${\cal N} P B_P = V_r^{-*} D_a^* V_l^{-1} R B_R$. 
        Rearranging we get
        $D_a^{-*} V_r^{*} {\cal N} P B_P =  V_l^{-1} R B_R$, which is the same thing as
        $D_a^{-*} V_r^{*} {\cal N} V_r (V_r^{-1} P) B_P =  (V_l^{-1} R) B_R$.

        Now we plug in that $P = V_{r, 1:n_c} \wh{B}_P$ and $R = V_{l, 1:n_c} \wh{B}_R$. To this end, note the following identifty, 
        \begin{align}
            V_r^{-1} V_r
            =
            V_r^{-1} \begin{bmatrix}
                V_{r,1:n_c} & V_{r,n_c+1:n} 
            \end{bmatrix}
            =
            \begin{bmatrix}
                I_{n_c} & 0 \\
                0 & I_{n_f}
            \end{bmatrix},
        \end{align}
        which means that
        $
        V_r^{-1} P
        =
        V_r^{-1} V_{r, 1:n_c} \wh{B}_P
        =
        \begin{bmatrix}
            I \\
            0
        \end{bmatrix}
        \wh{B}_P.
        $
        By analogy, $V_l^{-1} R = \begin{bmatrix}
            I \\
            0
        \end{bmatrix}
        \wh{B}_R$.
        Plugging these into $D_a^{-*} V_r^{*} {\cal N} V_r (V_r^{-1} P) B_P =  (V_l^{-1} R) B_R$ means that \eqref{eq:calM-def} is equivalent to
        \begin{align}
            D_a^{-*} V_r^{*} {\cal N} 
            V_r 
            \begin{bmatrix}
                I_{n_c} \\
                0_{n_f \times n_c}
            \end{bmatrix}
            (\wh{B}_P B_P) 
            =  
            \begin{bmatrix}
                I_{n_c} \\
                0_{n_f \times n_c}
            \end{bmatrix} 
            (\wh{B}_R B_R).
        \end{align}
        Let us define $B_P := \wh{B}_P B_P$ and $B_R := \wh{B}_R B_R$.
        Now, we left multiply this equation by $D_a^{*}$, right multiply it by $B_P^{-1}$ and then define $
        \begin{bmatrix}
                B \\
                0_{n_f \times n_c}
        \end{bmatrix}
         \coloneqq 
        D_a^{*} \begin{bmatrix}
                B_R B_P^{-1} \\
                0_{n_f \times n_c}
        \end{bmatrix}
        $.
        Then the question of whether $\Pi(P_{\sharp}, R_{\sharp})$ is orthogonal in any inner product is equivalent to whether there exists an HPD matrix ${\cal N}$ and a non-singular $B$ such that
        \begin{align} \label{eq:calM-constraint}
            V_r^{*} {\cal N} V_r 
            \begin{bmatrix}
                I_{n_c} \\
                0_{n_f \times n_c}
            \end{bmatrix}
            =  
            \begin{bmatrix}
                B \\
                0_{n_f \times n_c}
            \end{bmatrix}.
        \end{align}

        Define ${\cal A}  \coloneqq  V_r^{*} {\cal N} V_r $. Notice that, ${\cal A}^* = V_r^{*} {\cal N}^* V_r = V_r^{*} {\cal N} V_r = {\cal A}$, so that ${\cal A}$ is Hermitian. Hence, we must have the block structure 
        ${\cal A} = 
        \begin{bmatrix}
        {\cal A}_{cc} & {\cal A}_{cf} \\ {\cal A}_{cf}^* & {\cal A}_{ff} 
        \end{bmatrix}$.
        Plugging into \eqref{eq:calM-constraint} yields
        \begin{align}
        {\cal A}
             \begin{bmatrix}
                I_{n_c} \\
                0_{n_f \times n_c}
            \end{bmatrix}
            =
            \begin{bmatrix}
                {\cal A}_{cc} \\
                {\cal A}_{cf}^*
            \end{bmatrix}
            =  
            \begin{bmatrix}
                B \\
                0_{n_f \times n_c}
            \end{bmatrix}.
        \end{align}
        That is, ${\cal A}  \coloneqq  V_r^{*} {\cal N} V_r$ must be CF-block diagonal, as stated in the lemma.
        Recall that $B$ must be non-singular, and, indeed, this is guaranteed because ${\cal A}_{cc} = B$ is in fact HPD under the assumption of ${\cal N}$ being HPD (see the reasoning in the proof of \cref{thm:orth-characterization}). 
        \end{proof}

        Observe that any ${\cal N}$ of the form ${\cal N} = V_r^{-*} \wt{{\cal D}}^* \wt{{\cal D}} V_r^{-1}$ satisfies the requirements of \cref{lem:orth-nec-and-suf}.
        We have $V_r^* {\cal N} V_r = (V_r^* V_r^{-*}) (\wt{{\cal D}}^* \wt{{\cal D}}) (V_r^{-1} V_r) = \wt{{\cal D}}^* \wt{{\cal D}}$ is block diagonal, as required.
        For $\Vert \cdot \Vert_{\cal N}$ to be a valid norm ${\cal N}$ must be HPD; clearly ${\cal N}$ of this form is Hermitian positive semi-definite, but in fact it is positive definite by assumption of ${\cal D}$ being full rank.

        Now we show that \textit{any} ${\cal N}$ satisfying \cref{lem:orth-nec-and-suf} can be written in the form of ${{\cal N} = V_r^{-*} \wt{{\cal D}}^* \wt{{\cal D}} V_r^{-1}}$.
        Recall that any valid ${\cal N}$ must by HPD, and it must satisfy that ${\cal A}  \coloneqq  V_r^* {\cal N} V_r$ is block diagonal.
        Since $V_r$ is non-singular, notice that ${\cal A}$ is congruent to ${\cal N}$, and hence ${\cal A}$ must be HPD under assumption of ${\cal N}$ being HPD \cite[Proposition 3.4.5]{Bernstein-2009}. 
        Since ${\cal A}$ is block diagonal, it is HPD iff its diagonal blocks are HPD. As such, \textit{any} ${\cal A}$ must have a Cholesky decomposition of the form 
        ${\cal A} 
        = 
        \begin{bmatrix}
            {\cal U}_{cc} & 0 \\ 0 & {\cal U}_{ff}
        \end{bmatrix}^*
        \begin{bmatrix}
            {\cal U}_{cc} & 0 \\ 0 & {\cal U}_{ff}
        \end{bmatrix}
        =
        \begin{bmatrix}
            {\cal A}_{cc} & 0 \\ 0 & {\cal A}_{ff}
        \end{bmatrix}
        $, where ${\cal U}_{cc}$ and ${\cal U}_{ff}$ are upper triangular matrices.
        Rearranging 
        ${\cal A} = V_r^* {\cal N} V_r$ 
        for ${\cal N}$ gives
        \begin{align}
            {\cal N}  
            =
            V_r^{-*} {\cal A} V_r^{-1}
            =
        \left(
            V_r^{-*}\begin{bmatrix}
            {\cal U}_{cc} & 0 \\ 0 & {\cal U}_{ff}
        \end{bmatrix}^*
        \right)
        \left(
        \begin{bmatrix}
            {\cal U}_{cc} & 0 \\ 0 & {\cal U}_{ff}
        \end{bmatrix}
        V_r^{-1}
        \right).
        \end{align}
        This ${\cal N}$ is in the form stated in the theorem.

        \end{proof}

In the following section we restrict our attention to norms induced by diagonal matrices $\mathcal{D}$ in \eqref{eq:calD-def} rather than the broader class of block-diagonal matrices $\widetilde{\mathcal{D}}$ in \eqref{eq:wtcalD-def} considered in \cref{thm:orth-characterization}, which allows us to prove two-level convergence bounds. 

\subsection{Two-level convergence in norm}\label{sec:norm-conv-opt:two-grid}

    The class of norms derived in the previous section measure how the two-level operator acts on generalized right eigenvectors of the pencil $(A,M)$, which is exactly the basis that the fine-space preconditioner operates on, and, by design, is exactly the basis that the range of the coarse-space correction is constructed with respect to in forming $P_\sharp$ and $R_\sharp$.
    To that end, since we understand how both the fine-space preconditioner and the coarse-space correction act in the ${\cal N}$-norm, it is straightforward to prove two-level convergence.
    The following theorem is a strengthening of the result in \cref{cor:optimality_transfer-complex}, i.e., \cite[Corollary 5.2]{Ali-etal-2024-optimalP-gen}, which proved two-level convergence of V$(1,0)$ cycles with respect to the spectral radius. 
    Note the spectral radius is not a matrix norm and merely serves as a lower bound on any valid matrix norm. Moreover, a spectral radius smaller than unity provides a guarantee of asymptotic convergence, but does not guarantee a contraction of the error in each iteration. 
    Here we derive tight two-level convergence bounds in the ${\cal N}$-norm for V$(\nu_1,\nu_2)$-cycles, which end up being equal to the spectral radius of the two-level error propagation operator \eqref{eq:two-grid}. 
    Finally, the following result establishes for any $k \in \mathbb{Z}_{> 0}$ that $\Vert [E_{\rm TG}^{\nu_1, \nu_2}(P_{\sharp},R_{\sharp})]^k \Vert_{\cal N}
    =
    \Vert E_{\rm TG}^{\nu_1, \nu_2}(P_{\sharp},R_{\sharp}) \Vert_{\cal N}^k$.
    This is an interesting observation because for a non-normal matrix $N$ it can be the case that $\Vert N^k \Vert \ll \Vert N \Vert^k$.
    That is, for a non-normal error propagator, one can observe divergence on initial iterations, only to see convergence on later iterations or even asymptotically.
    For example, we have observed this phenomenon with reduction-based multigrid algorithms applied to hyperbolic problems, such as AIR \cite{manteuffel2018nonsymmetric,manteuffel2019nonsymmetric}, and MGRIT \cite{DeSterck-etal-2025-LFA}, but evidently this is not possible when considering ${\cal N}$-norm convergence of $E_{\rm TG}^{\nu_1, \nu_2}(P_{\sharp},R_{\sharp})$.

    \begin{theorem} \label{thm:two-grid-conv}
        Let ${\cal N}=V_r^{-*}\mathcal{D}^*\mathcal{D} V_r^{-1}$, where ${\cal D}$ is any diagonal matrix of the form in \eqref{eq:calD-def}.
        Then, the spectral radius, {the geometrically averaged ${\cal N}$-norm}, and the ${\cal N}$-norm of the two-level error propagation operator \eqref{eq:two-grid} are equal, and given by
        \begin{equation}
        \label{eq:Etg-M-norm}
        \begin{aligned} 
        \rho( E_{\rm TG}^{\nu_1, \nu_2}(P_{\sharp},R_{\sharp}) )
        =
        \Vert [E_{\rm TG}^{\nu_1, \nu_2}(P_{\sharp},R_{\sharp})]^k \Vert_{\cal N}^{1/k} 
        &=
        \Vert E_{\rm TG}^{\nu_1, \nu_2}(P_{\sharp},R_{\sharp}) \Vert_{\cal N} 
        \\
        &=
        \begin{cases} 
        |1-\lambda_{n_{c}+1}|^{\nu_1+\nu_2}, & n_c < n, \\ 0, & n_c = n, 
        \end{cases}
        \end{aligned}
        \end{equation}
        where $k \in \mathbb{Z}_{> 0}$.
    \end{theorem}

    \begin{proof}
    Consider the $n_c = n$ case, corresponding to $\ran( P_{\sharp} ) = \ran( R_{\sharp} ) = \mathbb{C}^n$, such that the coarse-space correction eliminates all error; that is $[I - \Pi(P_{\sharp}, R_{\sharp}) ] \bm{e} = \bm{0}$ for any $\bm{e} \in \mathbb{C}^n$.
    The result \eqref{eq:Etg-M-norm} follows trivially in this case.
    Throughout the remainder of the proof assume that $n_c < n$.
    
    Note that for any ${\cal Z} \in \mathbb{C}^{n \times n}$ we have
    \begin{align} \label{eq:calM-norm-def-mat}
    \begin{split}
        \Vert {\cal Z} \Vert_{\cal N}
         = &
        \max_{ \bm{e} \neq \bm{0} }
        \frac{\Vert {\cal Z} \bm{e} \Vert_{\cal N}}{\Vert \bm{e} \Vert_{\cal N}}
        =
        \max_{ \bm{e} \neq \bm{0} }
        \frac{\Vert ({\cal D} V_r^{-1}) {\cal Z} \bm{e} \Vert}{\Vert ({\cal D} V_r^{-1}) \bm{e} \Vert}
        \\ &=
        \max_{ \bm{y} \neq \bm{0} }
        \frac{\Vert ({\cal D} V_r^{-1}) {\cal Z} ({\cal D} V_r^{-1})^{-1} \bm{y} \Vert}{\Vert \bm{y} \Vert}
         = 
        \Vert ({\cal D} V_r^{-1}) {\cal Z} ({\cal D} V_r^{-1})^{-1} \Vert.
    \end{split}
    \end{align}%
    Thus
    $\Vert E_{\rm TG}^{\nu_1, \nu_2} \Vert_{\cal N} 
    = 
    \Vert {\cal D} V_r^{-1} E_{\rm TG}^{\nu_1, \nu_2} V_r {\cal D}^{-1} \Vert$. 
    This proof works by developing an eigenvalue decomposition for $E_{\rm TG}^{\nu_1, \nu_2}$, with right eigenvectors $V_r$; see the proof of \cite[Theorem 5.1]{Ali-etal-2024-optimalP-gen} for related decomposition of $E_{\rm TG}^{1, 0}$.
    To this end, consider the following
    \begin{subequations}
    \begin{align}
        E_{\rm TG}^{\nu_1, \nu_2} (P_{\sharp},R_{\sharp}) V_r
        &=
        (I - M^{-1} A)^{\nu_2} [ I - \Pi (P_{\sharp},R_{\sharp}) ] (I - M^{-1} A)^{\nu_1}
        V_r
        \\
        &=
        (I - M^{-1} A)^{\nu_2} V_r V_r^{-1} [ I - \Pi (P_{\sharp},R_{\sharp}) ] V_r (I - \Lambda)^{\nu_1}
        \\ 
        \label{eq:Etg-partial-diagonalization}
        &=
        V_r (I - \Lambda)^{\nu_2} [ I -  V_r^{-1}  \Pi (P_{\sharp},R_{\sharp}) V_r ] (I - \Lambda)^{\nu_1}.
    \end{align}
    \end{subequations}
    Here we have used the fact that since $(I - M^{-1} A) V_r = V_r ( I - \Lambda )$, it must hold for any $\nu \in \mathbb{Z}_{\geq 0}$ that $(I - M^{-1} A)^{\nu} V_r = V_r ( I - \Lambda )^{\nu}$. 
    Now consider the coarse-space projection
    \begin{align}
        V_r^{-1}  \Pi (P_{\sharp},R_{\sharp}) V_r
        &=
        (V_r^{-1} P_{\sharp}) [(R_{\sharp}^* A V_r) (V_r^{-1} P_{\sharp})]^{-1} (R_{\sharp}^* A V_r).
    \end{align}
    Recall that $\Pi (P_{\sharp},R_{\sharp})$ is invariant to coarse-space basis changes to $P_{\sharp}$ and $R_{\sharp}$ (e.g., see \cref{lem:cgc-basis-change}), for simplicity, let us fix $P_{\sharp} = V_{r, 1:n_c}$ and $R_{\sharp} = V_{l, 1:n_c}$.
    Next, recall from the proof of \cref{lem:orth-nec-and-suf} that $V_r^{-1} P_{\sharp} = 
    \begin{bmatrix} I \\ 0 \end{bmatrix}$, and $R_{\sharp}^* A V_r = R_{\sharp}^* V_l^{-*} V_l^* A V_r = (V_l^{-1} R_{\sharp})^* ( V_l^* A V_r ) = \begin{bmatrix}
        I & 0
    \end{bmatrix} D_a$ for some diagonal matrix $D_a = \begin{bmatrix}
        D_{a, cc} & 0 \\ 0 & D_{a, ff}
    \end{bmatrix}$; see \cref{lem:orth-complex}.
    As such, we have
    \begin{subequations}\label{eq:Etg-sharp-eigen-decomp}
    \begin{align}
        V_r^{-1} \Pi (P_{\sharp},R_{\sharp}) V_r
        &=
    \begin{bmatrix}
        I \\ 0
    \end{bmatrix}
    \left(
        \begin{bmatrix}
        I & 0
        \end{bmatrix} 
        \begin{bmatrix}
        D_{a, cc} & 0 \\ 0 & D_{a, ff}
    \end{bmatrix}
    \begin{bmatrix} I \\ 0 \end{bmatrix}
    \right)^{-1}
    \begin{bmatrix}
        I & 0
    \end{bmatrix} D_a
    \\
    &=
    \begin{bmatrix}
        I \\ 0
    \end{bmatrix}
    \left(
        D_{a, cc} 
    \right)^{-1}
    \begin{bmatrix}
        I & 0
    \end{bmatrix} 
    \begin{bmatrix}
        D_{a, cc} & 0 \\ 0 & D_{a, ff}
    \end{bmatrix}
    =
    \begin{bmatrix}
        I & 0 \\ 0 & 0
    \end{bmatrix}.
    \end{align}
    \end{subequations}
    Plugging back into \eqref{eq:Etg-partial-diagonalization} we find the eigenvalue decomposition
    \begin{subequations}
    \begin{align}
        E_{\rm TG}^{\nu_1, \nu_2} (P_{\sharp},R_{\sharp}) V_r 
        &= 
        V_r
        \begin{bmatrix} 
        (I - \Lambda_{cc})^{\nu_2} & 0 \\ 0 & (I - \Lambda_{ff})^{\nu_2} 
        \end{bmatrix}
        \begin{bmatrix} 0 & 0 \\ 0 & I \end{bmatrix}
        \label{eq:Etg-sharp-eigen-decomp2}\begin{bmatrix} 
        (I - \Lambda_{cc})^{\nu_1} & 0 \\ 0 & (I - \Lambda_{ff})^{\nu_1} 
        \end{bmatrix}
        \\
        &=
        V_r
        \begin{bmatrix} 0 & 0 \\ 0 & (I - \Lambda_{ff})^{\nu_1 + \nu_2} \end{bmatrix}.
    \end{align}
    \end{subequations}
    Recalling that $I - \Lambda_{ff} = \diag( 1 - \lambda_{n_c+1}, \ldots, 1 - \lambda_n )$, where these eigenvalues are ordered such that $|1 - \lambda_{n_c+1}| \geq \cdots \geq |1 - \lambda_{n}|$ for all $k = 1, \ldots, n-1$, the spectral radius is given by
    \begin{align}
        \rho( E_{\rm TG}^{\nu_1, \nu_2} (P_{\sharp},R_{\sharp}) )
        =
        \rho( (I - \Lambda_{ff})^{\nu_1 + \nu_2} )
        =
        |1 - \lambda_{n_c+1}|^{\nu_1 + \nu_2}.
    \end{align}

    Now consider the ${\cal N}$-norm result. Recalling the ${\cal N}$-norm definition \eqref{eq:calM-norm-def-mat} and using the eigenvalue decomposition \eqref{eq:Etg-sharp-eigen-decomp2} and we have that
    \begin{subequations}
    \begin{align} 
        \Vert E_{\rm TG}^{\nu_1, \nu_2} (P_{\sharp},R_{\sharp}) \Vert_{\cal N}
        &=
        \Vert {\cal D} V_r^{-1} E_{\rm TG}^{\nu_1, \nu_2} (P_{\sharp},R_{\sharp}) V_r {\cal D}^{-1} \Vert
        \\
        &=
        \left\Vert 
        {\cal D} V_r^{-1} V_r
        \begin{bmatrix}
            0 & 0 \\
            0 & (I - \Lambda_{ff})^{\nu_1 + \nu_2}
        \end{bmatrix} 
        {\cal D}^{-1}
        \right\Vert
        \\
        \label{eq:calM-norm-penult}
        &=
        \Vert 
        {\cal D}_{ff}
         (I - \Lambda_{ff})^{\nu_1 + \nu_2}
        {\cal D}_{ff}^{-1}\Vert. 
    \end{align}
    \end{subequations}
    Since, by assumption, ${\cal D}_{ff}$ is diagonal, it commutes with the diagonal matrix \\${(I - \Lambda_{ff})^{\nu_1 + \nu_2}}$, and we get a norm equal to 
    $\Vert     
    (I - \Lambda_{ff})^{\nu_1 + \nu_2}
    \Vert 
    = 
    |1 - \lambda_{n_c + 1}|^{\nu_1 + \nu_2}$.

    Finally, consider the geometrically averaged ${\cal N}$-norm.
    For any matrix $E$, and any $k \in \mathbb{Z}_{> 0}$, observe the inequalities 
    \begin{align}
        \rho(E) 
        \leq 
        \Vert E^{k} \Vert_{{\cal N}}^{1/k} 
        \leq 
        \Vert E \Vert_{\cal N}.
    \end{align}
    Noting that $\Vert \cdot \Vert_{\cal N}$ is a valid matrix norm, the left-hand inequality is from Gelfand's theorem, and the right-hand inequality follows from submultiplicativity of the ${\cal N}$-norm; that is, ${ \Vert E^{k} \Vert_{{\cal N}}^{1/k} \leq \big( \prod_{j = 1}^k \Vert E \Vert_{{\cal N}} \big)^{1/k} = \Vert E \Vert_{\cal N} }$.   
    The geometrically averaged norm result in \eqref{eq:Etg-M-norm} then follows because the spectral radius and ${\cal N}$-norm of $E_{\rm TG}^{\nu_1, \nu_2}(P_{\sharp},R_{\sharp})$ are equal. This concludes the proof.

    \end{proof}

\subsection{Optimality of $P_{\sharp}$ and $R_{\sharp}$}\label{sec:norm-conv-opt:optimality}

        \Cref{thm:optimal} below establishes optimality of $P_{\sharp}$ and $R_{\sharp}$ with respect to two-level convergence in the ${\cal N}$-norm.
        We remark that the proof used here relies on a type of Courant-Fischer-Weyl min-max principle, and in this sense it is similar in spirit to that used in \cite{Brannick-etal-2018-optimalP} for proving optimality of $P_{\sharp}$ in the $A$-norm in the HPD setting.
        We also want to emphasize that this result establishes a \textit{genuine} norm-based optimality over the space of all possible transfer  operators, and therefore is much stronger than the pseudo-optimal spectral-radius-based optimality established in \cite{Ali-etal-2024-optimalP-gen}. We begin by stating the generalized Courant-Fischer-Weyl theorem (see, e.g., \cite{Li-2015}), and then state our main optimality result.

        \begin{theorem}[Generalized Courant-Fischer-Weyl min-max principle] \label{thm:courant}
            Let $\alpha_1 \leq \alpha_2 \leq \cdots \leq \alpha_n$ be the generalized eigenvalues of the Hermitian matrix pencil \\${({\cal A}, {\cal B}) \in \mathbb{C}^{n \times n}}$, with ${\cal B}$ positive definite.
            Then, for any valid subspace $T$,
            \begin{align} \label{eq:courant}
                \alpha_k = \min_{
            \substack{ T \subset \mathbb{C}^{n} \\ \mathrm{dim}(T) = n - k + 1} }
            \max_{\substack{ \bm{x} \in T \\ \bm{x} \neq \bm{0}} } \frac{\bm{x}^* {\cal A} \bm{x}}{\bm{x}^* {\cal B} \bm{x}},
            \quad
            k = 1, \ldots, n.
            \end{align}
        \end{theorem}

        \begin{theorem}\label{thm:optimal}
            Let ${\cal D}$ be any diagonal matrix of the form in \eqref{eq:calD-def}.
            Then, the transfer operators $P_{\sharp}$, and $R_{\sharp}$ minimize  the ${\cal N}=(V_r^{-*}\mathcal{D}^*\mathcal{D} V_r^{-1})$-norm of the two-level error propagation operator \eqref{eq:two-grid}.
            That is,
            \begin{align}
                (P_{\sharp}, R_{\sharp})
                =
                \argmin_{ P, R \in \mathbb{C}^{n \times n_c}}
                \Vert E_{\rm TG}^{\nu_1, \nu_2}(P, R) \Vert_{\cal N},
            \end{align}
            with the minimum-norm value given in \cref{thm:two-grid-conv}.
        \end{theorem}

        It is quite interesting to note that both the error propagation norm in \cref{thm:two-grid-conv} and the optimality result in \cref{thm:optimal} hold independently of the specific diagonal matrix ${\cal D}$ used in the ${\cal N}$-norm.
        To this end, before proving \cref{thm:optimal}, we present the following natural corollary pertaining to HPD $A$ and $M$. This corollary follows because the ${\cal N}$-norm considered in the above theorem and \cref{thm:two-grid-conv} includes both the $A$- and $M$-norms when $A$ and $M$ are HPD; see discussion at the beginning of \cref{sec:norm-conv-opt}.
        We note that the $A$-norm optimality of $P_{\sharp}$ proved in \cite{Brannick-etal-2018-optimalP} is essentially the $A$-norm result below for the special case of $(\nu_1,\nu_2)=(0,1)$. 

        \begin{corollary}
            Let $A$ and $M$ be HPD. Then,
            \begin{align}
                P_{\sharp}
                =
                \argmin_{ P \in \mathbb{C}^{n \times n_c}}
                \Vert E_{\rm TG}^{\nu_1, \nu_2}(P, P) \Vert_{A}
                =
                \argmin_{ P \in \mathbb{C}^{n \times n_c}}
                \Vert E_{\rm TG}^{\nu_1, \nu_2}(P, P) \Vert_{M},
            \end{align}
            where
            \begin{align}
                \Vert E_{\rm TG}^{\nu_1, \nu_2}(P_{\sharp}, P_{\sharp}) \Vert_{A}
                =
                \Vert E_{\rm TG}^{\nu_1, \nu_2}(P_{\sharp}, P_{\sharp}) \Vert_{M}
                =
                |1 - \lambda_{n_c + 1}|^{\nu_1 + \nu_2},
                \quad
                n_c < n.
            \end{align}
        \end{corollary}
    
        Finally, before proving \cref{thm:optimal} we present one more natural corollary on necessary and sufficient conditions for convergence.

        \begin{corollary}[Necessary and sufficient conditions]
            \label{cor:nec-and-suf}
            Consider two-level convergence in the ${\cal N}$-norm, i.e., $\Vert E_{\rm TG}^{\nu_1, \nu_2}(P, R) \Vert_{\cal N} < 1$, with $P, R \in \mathbb{C}^{n \times n_c}$, and $n_c < n$. 
            Then, $|1 - \lambda_{n_c + 1}| < 1$ is a \emph{necessary} and \emph{sufficient} condition that there exists a convergent two-level method (namely that with $(P,R) = (P_\sharp, R_\sharp)$). 
        \end{corollary}

        \begin{proof}
            For any $R, P \in \mathbb{C}^{n \times n_c}$, we have
            \begin{equation*}
            \begin{aligned}
                |1 - \lambda_{n_c + 1}|^{\nu_1 + \nu_2}
                =
                \Vert E_{\rm TG}^{\nu_1, \nu_2}(P_{\sharp}, R_{\sharp}) \Vert_{\cal N}
                &=
                \min_{ \wc{P}, \wc{R} \in \mathbb{C}^{n \times n_c}}
                \Vert E_{\rm TG}^{\nu_1, \nu_2}( \wc{P}, \wc{R} ) \Vert_{\cal N}
                \\
                &\leq
                \Vert E_{\rm TG}^{\nu_1, \nu_2}(P, R) \Vert_{\cal N},
            \end{aligned}
            \end{equation*}
            where the equality follows from \cref{thm:two-grid-conv}, and the inequality from \cref{thm:optimal} due to the optimality of $(P, R) = (P_{\sharp}, R_{\sharp})$.
            From the equalities it follows that 
            ${
            |1 - \lambda_{n_c + 1}| 
            <
            1}
            $
            is both necessary and sufficient for two-level convergence with $(P, R) = (P_{\sharp}, R_{\sharp})$.
            On the other hand, by the inequality, no other two-level method can converge faster than that with $(P, R) = (P_{\sharp}, R_{\sharp})$, so this is a necessary and sufficient condition for there to exist \emph{any} convergent two-level method. 
        \end{proof}

        \begin{proof}

        Throughout the proof, assume that $n_c < n$ since the optimality claim clearly holds when $n = n_c$, as per \cref{thm:two-grid-conv}.

        Recalling from \eqref{eq:calM-norm-def-mat} that $\Vert E_{\rm TG}^{ \nu_1, \nu_2 } \Vert_{\cal N} 
        = 
        \Vert ({\cal D} V_r^{-1}) E_{\rm TG}^{ \nu_1, \nu_2 } ({\cal D} V_r^{-1})^{-1} \Vert
        =
        \\
        \Vert {\cal D} V_r^{-1} E_{\rm TG}^{ \nu_1, \nu_2 } V_r {\cal D}^{-1} \Vert
        $, 
        consider the following based on error propagation in \eqref{eq:two-grid}
        \begin{subequations}
        \begin{align}
        {\cal D} V_r^{-1} E_{\rm TG}^{ \nu_1, \nu_2 } V_r {\cal D}^{-1}
        &= 
        {\cal D} V_r^{-1} (I - M^{-1} A)^{\nu_2} [I - \Pi(P, R)] (I - M^{-1} A)^{\nu_1} V_r {\cal D}^{-1}
        \\
        &= 
        (I - \Lambda)^{\nu_2} {\cal D} V_r^{-1} [I - P(R^* A P)^{-1} R^* A] V_r {\cal D}^{-1} (I - \Lambda)^{\nu_1}
        \\
        \label{Etg-calM-decomp-partial}
        \begin{split}
        &=
        (I - \Lambda)^{\nu_2} 
         [I - ({\cal D} V_r^{-1} P) \{ (R^* A V_r {\cal D}^{-1}) ({\cal D} V_r^{-1} P) \}^{-1} 
         \\
         &\hspace{36ex}
         (R^* A V_r {\cal D}^{-1})] 
         (I - \Lambda)^{\nu_1}.
         \end{split}
        \end{align}
        \end{subequations}
        To arrive at the second equality, first note that since $(I - M^{-1} A) V_r = V_r ( I - \Lambda )$, it must hold for any $\nu_1 \in \mathbb{Z}_{\geq 0}$ that $(I - M^{-1} A)^{\nu_1} V_r = V_r ( I - \Lambda )^{\nu_1}$, with an analogous result holding for $V_r^{-1} (I - M^{-1} A)^{\nu_2}$.  
        Secondly, note that the diagonal matrix 
        $( I - \Lambda )^{\nu}$ commutes with ${\cal D}$ and ${\cal D}^{-1}$ by assumption of ${\cal D}$ being diagonal.
        Now we define 
        \begin{align} \label{eq:PRQ-hat-def}
        \wh{P}  \coloneqq   {\cal D} V_r^{-1} P,
        \quad
        \wh{R}^*  \coloneqq  R^* A V_r {\cal D}^{-1},
        \quad
        Q_{ \wh{P}, \wh{R} }  \coloneqq  
        \wh{P}
        \big( \wh{R}^* \wh{P} \big)^{-1} \wh{R}^*,
        \end{align}
        so that, plugging into \eqref{Etg-calM-decomp-partial} yields the decomposition
        \begin{align}
        \label{eq:Enu0-decomp}
        {\cal D} V_r^{-1} E_{\rm TG}^{ \nu_1, \nu_2 } V_r {\cal D}^{-1}
        &=
        (I - \Lambda)^{\nu_2}
        (I - Q_{ \wh{P}, \wh{R} } ) 
        (I - \Lambda)^{\nu_1},
        \end{align}

        Now, to minimize $\Vert E_{\rm TG}^{\nu_1, \nu_2}(P, R) \Vert_{\cal N}$ over $P$ and $R$, we first need to understand its dependence on $P$ and $R$.
        From above, we observe that $E_{\rm TG}^{ \nu_1, \nu_2 }$ depends on $P$ and $R$ only through the oblique projection $Q_{ \wh{P}, \wh{R} }$. 
        This projection is fully specified by the combination of its range and its null space, since any projection can be written as $C (B^* C)^{-1} B^*$, with range equal to that of $C$, and null space equal to that of $B^*$.
        So, from \eqref{eq:PRQ-hat-def},
        \begin{align}
            \ran( Q_{ \wh{P}, \wh{R} } )
            =
            \ran(\wh{P}),
            \quad
            \nul( Q_{ \wh{P}, \wh{R} } )
            =
            \nul( \wh{R}^* ).
        \end{align}
        From \eqref{eq:Enu0-decomp}, we actually need to consider the projection $I - Q_{ \wh{P}, \wh{R} }$. Noting from, e.g., \cite[Sec. 2.1]{Southworth-Manteuffel-2024-AMG} that $\ran( I - Q_{ \wh{P}, \wh{R} } )
        =
        \nul( Q_{ \wh{P}, \wh{R} } )$ 
        and 
        $\nul( I - Q_{ \wh{P}, \wh{R} } )
        =
        \ran( Q_{ \wh{P}, \wh{R} } )$, 
        we have:
        \begin{align}
            \ran( I - Q_{ \wh{P}, \wh{R} } )
            =
            \nul(\wh{R}^*),
            \quad
            \nul( I - Q_{ \wh{P}, \wh{R} } )
            =
            \ran(\wh{P}).
        \end{align}
        What this means is that the only dependence $E_{\rm TG}^{ \nu_1, \nu_2 }$ has on $\wh{P}$ is in the form of $\ran(\wh{P})$, and the only dependence it has on $\wh{R}$ is in the form $\nul(\wh{R}^*)$.
        Finally, since the matrices ${\cal D}, V_r^{-1},$ and $A$ in \eqref{eq:PRQ-hat-def} are fixed, ultimately $\ran(\wh{P})$ is determined by $\ran(P)$, and similarly $\nul( \wh{R}^* )$ is determined by $\nul(R^*)$, analogously to how the coarse-space correction $\Pi$ is invariant to coarse-space change of bases on $P$ and $R$ (see \cref{lem:cgc-basis-change}).

        Based on the above discussion, we therefore have the following equalities:
        \begin{subequations}
        \begin{align}
            \min_{P, R \in \mathbb{C}^{n \times n_c}}
            \Vert E_{\rm TG}^{ \nu_1, \nu_2 } \Vert_{\cal N}
            &=
            \min_{P, R \in \mathbb{C}^{n \times n_c}}
            \Vert 
            (I - \Lambda)^{\nu_2}
            (I - Q_{ \wh{P}, \wh{R} } ) 
            (I - \Lambda)^{\nu_1}
            \Vert
            \\
            &=
            \min_{\wh{P}, \wh{R} \in \mathbb{C}^{n \times n_c}}
            \Vert 
            (I - \Lambda)^{\nu_2}
            (I - Q_{ \wh{P}, \wh{R} } ) 
            (I - \Lambda)^{\nu_1}
            \Vert
            \\
            &=
            \min_{\wh{P} \in \mathbb{C}^{n \times n_c}}
            \min_{\substack{
            \nul( \wh{R}^* ) \subset \mathbb{C}^{n} 
            \\ 
            \label{eq:Etg-min-rewrite}
            \mathrm{dim}(\nul( \wh{R}^* )) = n_f
            }
            }
            \Vert 
            (I - \Lambda)^{\nu_2}
            (I - Q_{ \wh{P}, \wh{R} } ) 
            (I - \Lambda)^{\nu_1}
            \Vert.
        \end{align}
        \end{subequations}
        Since our goal is to minimize $\Vert E_{\rm TG}^{ \nu_1, \nu_2 } \Vert_{\cal N}^2$, we now develop a lower bound on that resembles \eqref{eq:Etg-min-rewrite}:
        \begin{align*}
        \Vert E_{\rm TG}^{ \nu_1, \nu_2 } \Vert_{\cal N}^2
        & =
        \Vert {\cal D} V_r^{-1} E_{\rm TG}^{ \nu_1, \nu_2 } V_r {\cal D}^{-1} \Vert^2
        =
        \max_{\bm{y} \neq \bm{0}}
        \frac{ \|
        (I - \Lambda)^{\nu_2} (I - Q_{ \wh{P}, \wh{R} } ) (I - \Lambda)^{\nu_1} \bm{y}\| 
        \| }{ \|\bm{y}\| }
        \\
        & =
        \max_{\bm{w} \neq \bm{0}}
        \frac{  
        \|(I - \Lambda)^{\nu_2} (I - Q_{ \wh{P}, \wh{R} } ) \bm{w}\|
        }{ \|(I - \Lambda)^{-\nu_1} \bm{w}\| }
        \geq
        \max_{\substack{\bm{w} \in \nul( \wh{R}^* ) \\ \bm{w} \neq \bm{0}}} 
        \frac{ \| 
        (I - \Lambda)^{\nu_2} (I - Q_{ \wh{P}, \wh{R} } ) \bm{w}
        \|}{ \|(I - \Lambda)^{-\nu_1} \bm{w} \| }
        \\
        &=
        \max_{\substack{\bm{w} \in \nul( \wh{R}^* ) \\ \bm{w} \neq \bm{0}}} 
        \frac{ \|
        (I - \Lambda)^{\nu_2} \bm{w}
        \| }{ \| (I - \Lambda)^{-\nu_1} \bm{w}\| }
        =
        \max_{\substack{\bm{w} \in \nul( \wh{R}^* ) \\ \bm{w} \neq \bm{0}}} 
        \frac{ \la  
        [(I - \Lambda) (I - \Lambda)^*]^{\nu_2}  \bm{w}
        \bm{w} 
        \ra }{ \la [(I - \Lambda) (I - \Lambda)^*]^{-\nu_1} \bm{w}, \bm{w} \ra }.
        \end{align*}
        Now let us use this inequality to lower bound the minimum of $\Vert E_{\rm TG}^{ \nu_1, \nu_2 } \Vert_{\cal N}^2$ with respect to $R$.
        We have
        \begin{subequations}
            \begin{align}
                \min_{R \in \mathbb{C}^{n \times n_c}}
                \Vert E_{\rm TG}^{ \nu_1, \nu_2 } \Vert_{\cal N}^2
                    &=
                \min_{\substack{
                \nul( \wh{R}^* ) \subset \mathbb{C}^{n} 
                \\ 
                \mathrm{dim}(\nul( \wh{R}^* )) = n_f
                }
                }
                \Vert 
                (I - \Lambda)^{\nu_2}
                (I - Q_{ \wh{P}, \wh{R} } ) 
                (I - \Lambda)^{\nu_1}
                \Vert
                \\
                \label{eq:E_tg_min_R_bound}
                &\geq
                \min_{\substack{
                \nul( \wh{R}^* ) \subset \mathbb{C}^{n} 
                \\ 
                \mathrm{dim}(\nul( \wh{R}^* )) = n_f
                }
                }
                \max_{\substack{\bm{w} \in \nul( \wh{R}^* ) \\ \bm{w} \neq \bm{0}}} 
                \frac{ \la  
                [(I - \Lambda) (I - \Lambda)^*]^{\nu_2}  \bm{w}
                \bm{w} 
                \ra }{ \la [(I - \Lambda) (I - \Lambda)^*]^{-\nu_1} \bm{w}, \bm{w} \ra },
                \\
                \label{eq:E_tg_min_R_bound2}
                &=
                |1 - \lambda_{n_c+1}|^{2(\nu_1 + \nu_2)}.
            \end{align}
        \end{subequations}
        The equality in \eqref{eq:E_tg_min_R_bound2} follows from the generalized min-max result in \cref{thm:courant}.
        Specifically, in \eqref{eq:E_tg_min_R_bound} we have the HPD matrix pencil $( [(I - \Lambda) (I - \Lambda)^*]^{\nu_2}, [(I - \Lambda) (I - \Lambda)^*]^{-\nu_1} )$.
        Since these two matrices are diagonal, the associated generalized eigenvalues are equivalent to the standard eigenvalues of the HPD matrix
        \begin{align}
            [(I - \Lambda) (I - \Lambda)^*]^{\nu_1 + \nu_2}
            =
            \diag( |1 - \lambda_1|^{2(\nu_1 + \nu_2)}
            \ldots,
            |1 - \lambda_n|^{2(\nu_1 + \nu_2)}).
        \end{align}
        Recalling that these eigenvalues satisfy the ordering $|1 - \lambda_1|^{2 (\nu_1 + \nu_2)} \geq \cdots \geq |1 - \lambda_{n}|^{2 (\nu_1 + \nu_2)}$ and then plugging into \eqref{eq:courant} that $n - k + 1 = n_f = n - n_c$ yields the eigenvalue index $k = n_c + 1$.

        Now we additionally minimize the inequality \eqref{eq:E_tg_min_R_bound2} over $P$ to get
        \begin{subequations}
            \begin{align}
            \min_{P, R \in \mathbb{C}^{n \times n_c}}
            \Vert E_{\rm TG}^{ \nu_1, \nu_2 } \Vert_{\cal N}^2
                &=
            \min_{\wh{P} \in \mathbb{C}^{n \times n_c}}
            \min_{\substack{
            \nul( \wh{R}^* ) \subset \mathbb{C}^{n} 
            \\ 
            \mathrm{dim}(\nul( \wh{R}^* )) = n_f
            }
            }
            \Vert 
            (I - \Lambda)^{\nu_2}
            (I - Q_{ \wh{P}, \wh{R} } ) 
            (I - \Lambda)^{\nu_1}
            \Vert
            \\
            &\geq
            \min_{\wh{P} \in \mathbb{C}^{n \times n_c}}
            |1 - \lambda_{n_c+1}|^{2(\nu_1 + \nu_2)}
            =
            |1 - \lambda_{n_c+1}|^{2(\nu_1 + \nu_2)}.
            \end{align}
        \end{subequations}
        Taking the square root on both sides, we see that the minimum value of $\Vert E_{\rm TG}^{\nu_1,\nu_2} \Vert_{\cal N}$ cannot be smaller than $|1 - \lambda_{n_c + 1}|^{\nu_1 + \nu_2}$, yet from \cref{thm:two-grid-conv} we know that this value is attained with $(P, R) = (P_{\sharp}, R_{\sharp})$. 
        As such, $(P, R) = (P_{\sharp}, R_{\sharp})$ must be minimizers as per the theorem statement.
        
        \end{proof}

\section{Real-valued optimal transfer operators}
\label{sec:real-valued}

The optimal transfer operators $P_{\sharp}$ and $R_{\sharp}$ discussed thus far are quite general in terms of the pencils $(A, M)$ that they can be derived for. 
However, one practical issue is that for real-valued matrix pencils $(A,M)$ with nonsymmetric or indefinite $A$ and $M$, the generalized eigenvectors and corresponding optimal transfer operators are almost certainly complex valued. Thus, while the results from \cref{sec:norm-conv-opt} provide useful theoretical information, the underlying transfer operators lack practical utility, since constructing complex-valued transfer operators for real-valued matrices is unappealing, and is likely not easily facilitated by many software libraries. 
To this end, here we present a modification to the sequence of results in \cref{sec:review} and \cref{sec:norm-conv-opt} that instead provide identical convergence bounds using real-valued restriction and interpolation operators in the event that $A$ and $M$ are real valued. 
We begin by providing a real-valued modification of the matrix-induced orthogonality results from \cite[Lemma 4.1]{Ali-etal-2024-optimalP-gen} (as stated in \Cref{lem:orth-complex}), analogous to, e.g., the real Schur decomposition, including a change of basis between the real and complex transfer operators. 

\begin{theorem}[Real-valued generalized eigenvalue decomposition]\label{theorem:orth-real}
    Let $A,M\in\mathbb{R}^{n\times n}$ be such that $M$ is invertible and $M^{-1}A$ is diagonalizable. Define the real-valued left and right generalized eigenvectors, $W_l$ and $W_r$, respectively, of the matrix pencil $(A,M)$, such that
    \begin{subequations}\label{eq:gep-real}
    \begin{align}
        AW_r & = MW_r\Lambda_{\mathbb{R}}, \\
       \ W_l^\top A & = \Lambda_{\mathbb{R}} W_l^\top M,\label{eq:gep-real-left}
    \end{align}
    \end{subequations}
    where $\Lambda_{\mathbb{R}}$ is a real-valued block-diagonal matrix representing eigenvalues, where real eigenvalues live on the diagonal, and complex-conjugate pairs of eigenvalues of the form $\eta\pm\mathrm{i}\zeta$, $\eta, \zeta \in \mathbb{R}$, are represented by a $2\times 2$ diagonal block 
    $\begin{bmatrix}
        \eta & -\zeta\\ \zeta & \eta
    \end{bmatrix}$.
    It is assumed that eigenvectors corresponding to conjugate eigenvalue pairs make up subsequent columns of $W_l$ and $W_r$. Then, $W_l$ and $W_r$ induce a matrix-based subspace orthogonality, satisfying
   \begin{subequations} \label{eq:orthog-real}
    \begin{align}
       W_l^\top AW_r &= {D}_{\mathbb{R}a}, \\
       W_l^\top MW_r &= {D}_{\mathbb{R}m},
    \end{align}
    \end{subequations}
    for block-diagonal matrices ${D}_{\mathbb{R}a},{D}_{\mathbb{R}m}$, with $2\times 2$ nonzero diagonal blocks corresponding to conjugate eigenvalue pairs and scalar diagonal corresponding to real eigenvalues. 
    In addition, if $V_{l,i:i+1} = [\bm{v}_{l} \,\, \cc{\bm{v}_{l}}]$ and $V_{r,i:i+1} = [\bm{v}_{r} \,\, \cc{\bm{v}_{r}}]$ denote a conjugate pair of left and right complex-valued generalized eigenvectors of $(A, M)$, from the  decomposition \eqref{eq:gep-complex}, we have the mapping
   \begin{subequations}\label{eq:gen-evec-complex-to-real}
    \begin{align}
        \bm{w}_{l,i} & = \Re(\bm{v}_{l}) + \Im(\bm{v}_{l}), \hspace{6ex}
        \bm{w}_{r,i} = \Re(\bm{v}_{r}) + \Im(\bm{v}_{r}), \\
        \bm{w}_{l,i+1} & = \Re(\bm{v}_{l}) - \Im(\bm{v}_{l}), \hspace{4ex}
        \bm{w}_{r,i+1} = \Re(\bm{v}_{r}) - \Im(\bm{v}_{r}). 
    \end{align}
    \end{subequations}
\end{theorem}

\begin{proof}
Let us start with the complex-valued generalized eigenvalue problem, as defined in \eqref{eq:gep-complex}. That is, let $V_{l}$ and $V_{r}$ denote the complex-valued left and right generalized eigenvectors, respectively, of the matrix pencil $(A,M)$, and $\Lambda$ a diagonal matrix of complex eigenvalues. 
By definition,
\begin{subequations}\label{eq:complex-gep-copy}
\begin{align}
    AV_{r} & = MV_{r}\Lambda, \\
    \label{eq:complex-gep-left}
    V_{l}^*A & = \Lambda V_{l}^*M.
\end{align}
\end{subequations}
Following \cref{lem:orth-complex} above, i.e., \cite[Lemma 4.1]{Ali-etal-2024-optimalP-gen}, $V_{l}$ and $V_{r}$ induce a matrix-based orthogonality, satisfying
\begin{subequations} \label{eq:orthog-complex-copy}
\begin{align}
   V_{l}^*AV_{r} &= D_{a}, \\
   V_{l}^*MV_{r} &= D_{m},
\end{align}
\end{subequations}
for complex-valued diagonal matrices $D_{a},D_{m}$. 

Now let us map the diagonal complex-valued eigenvalue matrix $\Lambda$ to the block-diagonal real-valued matrix $\Lambda_{\mathbb{R}}$ via the similarity transform ${\cal T}\Lambda {\cal T}^{-1} = \Lambda_{\mathbb{R}}$, where ${\cal T} \in \mathbb{C}^{n \times n}$ is a block diagonal matrix specified below.
Recall that for real valued $A$ and $M$, the generalized eigenvalues of $(A, M)$ are either real, or complex, arising in conjugate pairs. For real generalized eigenvalues, we simply let the diagonal of ${\cal T}$ be one. 
Now consider conjugate pairs of eigenvalues $\eta \pm \mathrm{i} \zeta$, with corresponding eigenvalue matrix $\Lambda_{i:i+1,i:i+1} = \begin{bmatrix} \eta + \mathrm{i} \zeta & 0 \\ 0 &  \eta - \mathrm{i} \zeta \end{bmatrix}$.
Then, consider the following similarity transform,
\begin{equation} \label{eq:calTi-def}
\begin{aligned}
    {\cal T}_i \Lambda_{i:i+1,i:i+1} {\cal T}_i^{-1}
    &= 
    \frac{1}{2}\begin{bmatrix} 1-\mathrm{i} & 1+\mathrm{i} \\ 1+\mathrm{i} & 1-\mathrm{i} \end{bmatrix}
         \begin{bmatrix} \eta + \mathrm{i} \zeta & 0 \\ 0 &  \eta - \mathrm{i} \zeta \end{bmatrix}
    \frac{1}{2}\begin{bmatrix} 1+\mathrm{i} & 1-\mathrm{i} \\ 1-\mathrm{i} & 1+\mathrm{i} \end{bmatrix} 
    \\
    &= \begin{bmatrix} \eta & -\zeta \\ \zeta & \eta \end{bmatrix}
    =
    \Lambda_{_{\mathbb{R}}i:i+1,i:i+1}.
\end{aligned}
\end{equation}
Thus, for pairs of rows in $\Lambda$ corresponding to conjugate pairs, we simply let the corresponding block diagonal of ${\cal T}$ be ${\cal T}_i$.
Note that in all cases, we have ${\cal T}^{-1} = {\cal T}^*$, which yields $\Lambda = {\cal T}^*\Lambda_{\mathbb{R}}{\cal T}$. Then define
\begin{equation}\label{eq:real-gen-evec}
    W_r\coloneqq V_{r}{\cal T}^{*}, \hspace{4ex} W_l \coloneqq V_{l}{\cal T}^{*}.
\end{equation}
Since ${\cal T}^*$ is block diagonal, the effect of the column scalings $V_{r}{\cal T}^{*}$ and $V_{l}{\cal T}^{*}$ in \eqref{eq:real-gen-evec} is to scale the $i$th conjugate pair of eigenvectors in $V_{r}$ and $V_{l}$ by the $i$th diagonal block ${\cal T}_i^*$ (or unity scaling for real-valued eigenvalues). 
Note that for arbitrary $a+\mathrm{i}b$, $a,b\in\mathbb{R}$, the product
\begin{equation}\label{eq:Ti-map-to-real}
    \begin{bmatrix} a + \mathrm{i}b & a + \mathrm{i}b \end{bmatrix}
    \frac{1}{2}\begin{bmatrix} 1+\mathrm{i} & 1-\mathrm{i} \\ 1-\mathrm{i} & 1+\mathrm{i} \end{bmatrix}
    =
    \begin{bmatrix} a + b & a  - b \end{bmatrix}
\end{equation}
is real valued.
Recall that conjugate pairs of left and right eigenvectors are themselves conjugate of each other, i.e., $\bm{v}_{l,i} = \cc{\bm{v}_{l,i+1}}$, and $\bm{v}_{r,i} = \cc{\bm{v}_{r,i+1}}$. 
Then \eqref{eq:Ti-map-to-real} implies that by the definition in \eqref{eq:real-gen-evec}, the transformed generalized left and right eigenvectors, $W_l$ and $W_r$, are necessarily real valued.
Substituting into \eqref{eq:complex-gep-copy} yields
\begin{subequations}
\begin{align}
    \begin{split}
        AV_{r} & = MV_{r}{\cal T}^{*}\Lambda_{\mathbb{R}}{\cal T} \\
        V_{l}^*A & = {\cal T}^{*}\Lambda_{\mathbb{R}}{\cal T}V_{l}^*M
    \end{split}
    \hspace{4ex}
    \mapsto 
    \hspace{-14ex}
     \begin{split}
        AW_{r} & = MW_{r}\Lambda_{\mathbb{R}} \\
        W_{l}^\top A & = \Lambda_{\mathbb{R}} W_{l}^\top M,
    \end{split}
\end{align}
\end{subequations}
which completes the proof of \eqref{eq:gen-evec-complex-to-real}.

Returning to the orthogonality relations \eqref{eq:orthog-complex-copy}, applying a similarity transformation to each in ${\cal T}$, we have
\begin{subequations}
\begin{align}
   \begin{split}
        {\cal T} V_{l}^* AV_{r}{\cal T}^* &= {\cal T} D_{a} {\cal T}^* \\
        {\cal T} V_{l}^* M V_{r}{\cal T}^* &= {\cal T} D_{m} {\cal T}^*
    \end{split}   
    \hspace{4ex}\mapsto 
    \hspace{-8ex}
    \begin{split}
        W_{l}^\top AW_{r} &= {\cal T} {D}_{a}{\cal T}^* \\
        W_{l}^\top MW_{r} &= {\cal T} {D}_{m} {\cal T}^*.
    \end{split}   
\end{align}
\end{subequations}
Now, because $W_l,W_r,A,M$ are all real valued, the right-hand side operators \\${{\cal T} D_{a} {\cal T}^* =: {D}_{\mathbb{R}a}}$ and ${\cal T} D_{m} {\cal T}^* =: {D}_{\mathbb{R}m}$ are necessarily real valued. Moreover, because ${\cal T},D_{a},D_{m}$ are block diagonal, with $2\times 2$ blocks for complex-conjugate generalized eigenvalue pairs and $1\times 1$ blocks for real generalized eigenvalues, we have that $W_l$ and $W_r$ obey a matrix-induced subspace orthogonality, as in \eqref{eq:orthog-real}. Specifically, the left and right real-valued generalized eigenvectors corresponding to a given real eigenvalue or conjugate pair of eigenvalues are $A$- and $M$-orthogonal to all other left and right real-valued generalized eigenvectors. This completes the proof.
\end{proof}

\subsection{Optimality of real-valued two-level components}

In this section we show that the optimality properties from \cref{sec:norm-conv-opt} of the complex-valued transfer operators $P_{\sharp}, R_{\sharp}$ carry over to real-valued analogs $P_{\mathbb{R}\sharp}, R_{\mathbb{R}\sharp}$. 
The key insight used to show this is that the coarse-space correction is invariant with respect to change of bases on the coarse space (see \cite[Sec. 2.2]{Manteuffel-Southworth-2019-NS-AMG}), as in the following lemma.

\begin{lemma} \label{lem:cgc-basis-change}
Define interpolation and restriction $P, R \in \mathbb{C}^{n \times n_c}$, and assume \\${R^*AP}$ is invertible. Consider the coarse-space projection
\begin{equation}
    \Pi(P, R)  \coloneqq P (R^* A P)^{-1} R^* A.
\end{equation}
Let $B_P, B_R \in \mathbb{C}^{n_c \times n_c}$ be invertible change-of-basis matrices.
Then, coarse-space correction is invariant with respect to the change of bases, $P \gets PB_P$ and $R \gets R B_R$; that is
\begin{equation}
    \Pi(P, R)
    =
    \Pi(P B_P, R B_R).
\end{equation}
Additionally, two-level error propagation \eqref{eq:two-grid} is also invariant under change of basis,
\begin{equation}
    E_{\rm TG}^{\nu_1,\nu_2}(P,R) = E_{\rm TG}^{\nu_1,\nu_2}(PB_P,RB_R).
\end{equation}
\end{lemma}
\begin{proof}
    \begin{align*}
        & \Pi(P B_P , R B_R)
        = (P B_P ) [ (R B_R)^* A (P B_P ) ]^{-1} (R B_R)^* A
        \\
        &\quad= P B_P  [ B_R^* R^* A P B_P  ]^{-1} B_R^*R^* A
        = P [R^* A P]^{-1} R^* A
        = \Pi (P, R).
    \end{align*}  
    The statement on two-level error propagation follows immediately.
\end{proof}

Recall that the optimal transfer operators in \cref{cor:optimality_transfer-complex}, defined with respect to complex-valued matrices $A$ and $M$, are complex valued, stemming from the fact that their ranges are defined in terms of the complex-valued generalized eigenvector decomposition of $(A, M)$ in \eqref{eq:gep-complex}. However, \cref{theorem:orth-real} establishes that when $A$ and $M$ are real valued there exists an analogous real-valued generalized eigenvalue decomposition of $(A, M)$. 
To this end, let $B_c \coloneqq (\mathcal{T}^*)_{cc}$, with ${\cal T}$ as in the proof of \cref{theorem:orth-real}, be a coarse-space change of basis matrix. 
Then, noting from \eqref{eq:real-gen-evec} that $W_l,W_r$ are defined via a column scaling of $V_l,V_r$ by ${\cal T}^*$, and because ${\cal T}^*$ is block diagonal we can define $P_{\mathbb{R}\sharp},R_{\mathbb{R}\sharp}\in\mathbb{R}^{n\times n_c}$ via their ranges as follows
\begin{equation}
    \begin{aligned}
    \ran(P_{\mathbb{R}\sharp}) &\coloneqq  \ran(P_{\sharp}B_c) = \ran(W_{r,1:n_c}), 
    \\
    \ran(R_{\mathbb{R}\sharp}) &\coloneqq \ran(R_{\sharp}B_c) = \ran(W_{l,1:n_c}).
    \end{aligned}
\end{equation}
Since the coarse-space correction is invariant to coarse-space change of bases on $P$ and $R$, as per \cref{lem:cgc-basis-change}, all convergence bounds and optimality from \cref{sec:norm-conv-opt} for complex transfer operators $\{P_\sharp,R_\sharp\}$ apply equivalently to their real-valued analogs $\{P_{\mathbb{R}\sharp},R_{\mathbb{R}\sharp}\}$ associated with the real-valued matrix pencil $(A,M)$. 
There are some nuances, however, which we now explain.

Recall from \eqref{eq:real-gen-evec} that the real-valued generalized right eigenvectors $W_r$ are defined such that ${W_r := V_r {\cal T}^*}$. 
So, considering the matrix ${\cal N}$ defined in \cref{thm:orth-characterization} we have
\begin{align} \label{eq:calM-def-Wr}
    {\cal N} 
    = 
    V_r^{-*} \wt{{\cal D}}^* \wt{{\cal D}} V_r^{-1}
    = 
    W_r^{-*} {\cal T}^{-*} \wt{{\cal D}}^* \wt{{\cal D}} {\cal T}^{-1} W_r^{-1}
    =
    W_r^{-\top} ({\cal T} \wt{{\cal D}}^* \wt{{\cal D}} {\cal T}^{*}) W_r^{-1}.
\end{align}
Recall from \eqref{eq:wtcalD-def}, and the proof of \cref{thm:orth-characterization}, that $\wt{{\cal D}}$ is such that $\wt{{\cal D}}^* \wt{{\cal D}}$
is an arbitrary, CF-split, block diagonal matrix with HPD blocks.
Therefore, the matrix ${\cal N}$ in \eqref{eq:calM-def-Wr} will not be real valued unless $\wt{{\cal D}}$ possesses the structure needed to make  ${\cal T} \wt{{\cal D}}^* \wt{{\cal D}} {\cal T}^{*}$ real valued.
Since the spirit of this section is to develop real-valued optimal transfer operators, the convergence and optimality results presented in \cref{thm:real-valued} below are restricted to a space of norms induced by a real-valued matrix $\wh{{\cal N}} \in \mathbb{R}^{n \times n}$.

To this end, let us restrict our attention to the class of matrices in \eqref{eq:calM-def-Wr} with diagonal matrices $\wt{{\cal D}} = {\cal D}$ defined in \eqref{eq:calD-def}, noting that \cref{thm:two-grid-conv,thm:optimal} apply specifically to this diagonal case.
Now, consider ${\cal T} {\cal D}^* {\cal D} {\cal T}^*$, noting that ${\cal D}^* {\cal D} \in \mathbb{R}^{n \times n}$ is a positive diagonal matrix. 
From the proof of \cref{theorem:orth-real} recall that ${\cal T}$ is block diagonal, with either a $1 \times 1$ block equal to unity, or a $2 \times 2$ diagonal block equal to ${\cal T}_i$, (implicitly) defined in \eqref{eq:calTi-def}. 
As such, ${\cal T} {\cal D}^* {\cal D} {\cal T}^*$ is a block diagonal matrix with block structure equal to that of ${\cal T}$.
The $1 \times 1$ diagonal blocks of ${\cal T} {\cal D}^* {\cal D} {\cal T}^*$ are just equal to the corresponding diagonal entries in ${\cal D}^* {\cal D}$.
Now consider the $2 \times 2$ case, corresponding to rows and columns $i:i+1$ in ${\cal T}$.
Suppose that $({\cal D}^* {\cal D})_{i:i+1, i:i+1}=\begin{bmatrix} a & 0 \\ 0 & b \end{bmatrix}$ for some $a, b \in \mathbb{R}_+$.
Then,
\begin{equation}
\begin{aligned}
    \left( {\cal T} {{\cal D}}^* {{\cal D}} {\cal T}^* \right)_{i:i+1,i:i+1}
    &=
    \frac{1}{2}\begin{bmatrix} 1-\mathrm{i} & 1+\mathrm{i} \\ 1+\mathrm{i} & 1-\mathrm{i} \end{bmatrix}
 \begin{bmatrix} a & 0 \\ 0 & b \end{bmatrix}
\frac{1}{2}\begin{bmatrix} 1+\mathrm{i} & 1-\mathrm{i} \\ 1-\mathrm{i} & 1+\mathrm{i} \end{bmatrix}
\\
&=
\frac{1}{2}
\begin{bmatrix}
    (a + b) & \mathrm{i} (b - a) \\ \mathrm{i} (a - b) & (a + b)
\end{bmatrix}.
\end{aligned}
\end{equation}
Evidently, this block is not real unless $a = b$, in which case it simplifies to $\begin{bmatrix}
    a & 0 \\ 0 & a
\end{bmatrix}.
$
Since we want to consider real valued ${\cal N}$, and hence real valued ${\cal T} {{\cal D}}^* {{\cal D}} {\cal T}^*$, we consider norms induced by SPD matrices $\wh{{\cal N}} := W_r^{-\top} \wh{{\cal D}} W_r^{-1}$, where $\wh{{\cal D}}$ is any positive diagonal matrix of the form 
\begin{align} \label{eq:whcal-D-def}
    \wh{{\cal D}} = \diag(d_1, \ldots, d_n),
    \quad
    \textrm{where}
    \quad
    d_i \in \mathbb{R}_+,
    \quad
    \textrm{and}
    \quad
    d_{i+1} = d_{i}
    \quad
    \textrm{if}
    \quad
    \lambda_{i+1} = \cc{\lambda_{i}}.
\end{align}
With this, we are now ready to present our final result.

\begin{theorem} \label{thm:real-valued}
    Let $A,M\in\mathbb{R}^{n\times n}$. Define the real-valued optimal interpolation and restriction such that
    \begin{align*}
        \operatorname{range}(P_{\mathbb{R}\sharp}) &= \operatorname{range}\left(
        \begin{bmatrix} \bm{{w}}_{r,1} & \bm{{w}}_{r,2} & \cdots & \bm{{w}}_{r,n_{c}} \end{bmatrix}\right),\\
        \operatorname{range}(R_{\mathbb{R}\sharp}) &= \operatorname{range}\left(
        \begin{bmatrix} \bm{{w}}_{l,1} & \bm{{w}}_{l,2} & \cdots & \bm{{w}}_{l,n_{c}} \end{bmatrix}\right),
    \end{align*}
    where $\bm{{w}}_{r,i} = W_{r,i}$, and $\bm{{w}}_{l,i} = W_{l,i}$ are the $i$th real-valued generalized right and left eigenvectors, respectively, from decomposition \eqref{eq:gep-real}. 
    Then:
    
    \begin{enumerate}
        \item  Let ${\cal N}$ be HPD. The projection $\Pi(P_{\mathbb{R}\sharp}, R_{\mathbb{R}\sharp})$ is ${\cal N}$-orthogonal if and only if ${\cal N}$ can be written in the form ${\cal N} = W_r^{-\top} ({\cal T} \wt{{\cal D}}^* \wt{{\cal D}} {\cal T}^{*}) W_r^{-1}$, where $\wt{{\cal D}}$ is any block-diagonal matrix of the form in \eqref{eq:wtcalD-def}.

        \item Let $\wh{{\cal N}} = W_r^{-\top} \wh{{\cal D}} W_r^{-1}$ where $\wh{{\cal D}}$ is any diagonal matrix of the form in \eqref{eq:whcal-D-def}. 
        Then, the spectral radius, the geometrically averaged ${\cal N}$-norm and the ${\cal N}$-norm of the two-level error propagation operator \eqref{eq:two-grid} are equal, and given by
        \begin{equation} \label{eq:real-Etg-M-norm}
        \begin{aligned}
        &\rho( E_{\rm TG}^{\nu_1, \nu_2}(P_{\mathbb{R}\sharp},R_{\mathbb{R}\sharp}) ) 
        =
        \Vert [ E_{\rm TG}^{\nu_1, \nu_2}(P_{\mathbb{R}\sharp},R_{\mathbb{R}\sharp}) ]^k \Vert_{\wh{\cal N}}^{1/k}
        =
        \\
        &\Vert E_{\rm TG}^{\nu_1, \nu_2}(P_{\mathbb{R}\sharp},R_{\mathbb{R}\sharp}) \Vert_{\wh{\cal N}} 
        =
        \begin{cases} 
        |1-\lambda_{n_{c}+1}|^{\nu_1+\nu_2}, & n_c < n, \\ 0, & n_c = n, 
        \end{cases}
        \end{aligned}
        \end{equation}
        where $k \in \mathbb{Z}_{> 0}$.

        \item Let $\wh{{\cal D}}$ be any diagonal matrix of the form in \eqref{eq:whcal-D-def}. 
        Then, the transfer operators $P_{\mathbb{R}\sharp}$, and $R_{\mathbb{R}\sharp}$ minimize the $\wh{{\cal N}} = (W_r^{-\top} \wh{{\cal D}} W_r^{-1})$-norm of the two-level error propagation operator \eqref{eq:two-grid}.
        That is,
        \begin{equation}
            (P_{\mathbb{R}\sharp}, R_{\mathbb{R}\sharp})
            =
            \argmin_{ P, R \in \mathbb{C}^{n \times n_c}}
            \Vert E_{\rm TG}^{\nu_1, \nu_2}(P, R) \Vert_{\wh{\cal N}},
        \end{equation}
        with the minimum-norm value given in \eqref{eq:real-Etg-M-norm}.
    \end{enumerate}
\end{theorem}

\section{Numerical results}
\label{sec:num}

In this section we present numerical results to firstly provide supporting evidence for the earlier theoretical results, and secondly to serve as a demonstrative exploration of what is possible from algebraic two-level methods applied to challenging, nonsymmetric discretized partial differential equation (PDE) problems.
Discretization and solver details are discussed in \cref{sec:num-disc-setup,sec:num-solver-setup}, respectively. An advection-reaction problem is considered in \cref{sec:num-adv-rea}, and then a wave-equation problem in \cref{sec:num-wave}.
For notational simplicity throughout this section, we write $E_{\rm TG}^{\nu_1,\nu_2} \equiv E_{\rm TG}^{\nu_1,\nu_2}(P_{\sharp}, R_{\sharp})$.
\subsection{PDE description and discretization setup}
\label{sec:num-disc-setup}

We consider two PDE problems, with all discretizations implemented through the finite-element library Firedrake \cite{FiredrakeUserManual}.
The first problem is a steady advection-reaction equation defined on a 2D unit square domain $\Omega = [0, 1]^2$, given by
\begin{align} \label{eq:adv-rea}
    \bm{b} \cdot \nabla u + c_0 u = 0,
\end{align}
for advected quantity $u(\bm{x})$, advecting velocity field $\bm{b}(\bm{x}) = (\cos(\pi y)^2, \cos(\pi x)^2)^{\top}$, and reaction coefficient $c_0(\bm{x}) = \alpha_0 + \alpha_1 \chi_I(x)\chi_I(y)$. Here, $\alpha_0 = 0.1$, $\alpha_1 = 0.9$, and $\chi_I$ denotes a 1D characteristic function which is equal to $1$ in the interval $I = [0.25, 0.75]$ and zero otherwise. Further, the domain's boundary $\partial \Omega$ is split with respect to $\bm{b}$ into inflow and outflow regions $\partial \Omega^+$ and $\partial \Omega^-$, respectively, and \label{eq:adv_rea} is equipped with inflow boundary conditions $u|_{\bm{x} \in \partial \Omega^+} = u^+\equiv 1$.

We discretize this problem with both standard streamline upwind Petrov-Galerkin ($\mathrm{CG}_p$) and interior penalty-based discontinuous Galerkin ($\mathrm{DG}_p$) methods \cite{kuzmin2010guide}, with the subscript $p$ corresponding to the underlying polynomial degree. The CG-based discretization is given by finding $u_h \in \{v \in \mathrm{CG}_p: v|_{\bm{x} \in \partial \Omega^+} = u^+\}$ such that
\begin{align} \label{eq:adv-rea-CG-def}
\langle \eta + \tau \bm{b} \cdot \nabla \eta, \bm{b} \cdot \nabla u_h + c_0 u_h \rangle = 0, && \forall \eta \in \mathring{\mathrm{CG}}_p,
\end{align}
where $\langle \cdot, \cdot \rangle$ denotes the $L^2$-inner product, and $\mathring{\mathrm{CG}}_p = \{v \in \mathrm{CG}_p: v|_{\bm{x} \in \partial \Omega^+} = 0\}$. Further, the SUPG parameter is set to $\tau = \delta x/(p |\bm{b}|)$, for local mesh size $\delta x$.
The DG-based discretization is given by finding $u_h \in \mathrm{DG}_p$ such that
\begin{equation} \label{eq:adv-rea-DG-def}
\begin{aligned}
&\langle \phi, c_0 u_h \rangle - \langle \nabla_h \phi, \bm{b} u_h\rangle + \int_\Gamma [\![\phi]\!]\{\bm{b} u_h\} \d S + \int_\Gamma \kappa_p |\bm{b}\cdot\bm{n}|[\![\phi]\!][\![u_h]\!] \d S \\
& \hspace{1cm} + \int_{\partial\Omega^+} (\bm{b} \cdot \bm{n})\phi u^+ \d S + \int_{\partial\Omega^-} (\bm{b} \cdot \bm{n})\phi u_h \d S = 0 & \forall \phi \in \mathrm{DG}_p,
\end{aligned}
\end{equation}
where $\Gamma$ denotes the set of interior mesh facets, and $[\![\cdot]\!]$, $\{\cdot\}$ denote the jump and normal component's average values across facets, respectively. The gradient $\nabla_h$ is evaluated cell-wise. The penalty parameter is set to $\kappa_p = 1$. Further, $\bm{n}$ denotes the outward normal unit vector at the domain's boundary. Note that a term of the form $\langle \nabla \cdot \bm{b}, \phi u\rangle$ has been dropped from \eqref{eq:adv-rea-DG-def} since our choice of $\bm{b}$ is divergence-free. Finally, we note that analytic expressions for $\bm{b}$ and $c_0$ are used in the discretization, which are then evaluated at quadrature points.

For both CG and DG we consider a computational mesh with $n_x \times n_y$ quadrilateral elements, where $n_x = n_y = 2 \cdot 2^{r}$ for refinement $r \in \mathbb{N}$.
The problem sizes we consider range from $n = 25$ DOFs up to $n = 2401$ DOFs. 

\begin{remark}[Diagonalizability]
The theory derived herein assumes the matrix pencil $(A,M)$ is diagonalizable. In most numerical PDE settings, we expect this to be the case. One example where this does not hold is upwind DG discretizations of inflow-outflow advection-reaction \eqref{eq:adv-rea}. There, the resulting sparse matrix is block triangular by element block, and thus the matrix $A$ is not diagonalizable \cite{manteuffel2019nonsymmetric,manteuffel2018nonsymmetric}. However, results here demonstrate that a small stabilization in the form of an interior penalty formulation, which introduces global (i.e., downwind) coupling to the discretization, is sufficient for diagonalizability. Moreover, numerical results below indicate that the theory provides good predictive power, even with respect to observed $\ell^2$ error and residual reduction, despite potentially poorly conditioned eignevector matrices.
\end{remark}

The second problem we consider is a time-dependent mixed wave equation defined again on a 2D domain $\Omega = [0, 1]^2$. It is given by
\begingroup
\begin{align}  \label{eq:wave}
    \frac{\partial u}{\partial t} + c \nabla \cdot \bm{p} = 0, \hspace{4ex}
    \frac{\partial \bm{p}}{\partial t} + \nabla u = \bm{0},
\end{align}
\endgroup
where $c$ is defined analogously to $c_0$ in \eqref{eq:adv-rea}, except with $I = [0.2, 0.8]$. The problem is equipped with initial and Dirichlet boundary conditions of the form $u|_{t=0} = 1 + 0.1 \exp(-[(x - 0.5)^2 + (y - 0.5)^2]/0.01)$, $\bm{p}|_{t=0} = \bm{0}$, and $u|_{\bm{x} \in \partial \Omega} = u_b \equiv 1$, respectively. %

We discretize this problem in space using the continuous Galerkin space $\mathrm{CG}_p$ for $u$, and its vectorized version of $\mathrm{CG}_p^2$ for $\bm{p}$. 
Further, in view of technicalities in our code implementation related to reordering DOFs for our solver strategy, we implement the boundary conditions for $u$ weakly. 
To advance the spatially discretized solution from time $t^n$ to $t^{n+1} := t^n + \delta t$, we use a midpoint rule. Altogether, this leads to a discretization of the form
\begingroup
\begin{subequations} \label{eq:wave-disc}
\begin{align} 
    &\langle \eta, u_h^{n+1} - u_h^n + \delta t c \nabla \cdot \bar{\bm{p}}_h \rangle + \delta t \int_{\partial\Omega} \kappa_p \eta (u_h - u_b) \d S = 0 &\forall \eta \in \mathrm{CG}_p, \\
    &\langle \bm{w},  \bm{p}_h^{n+1} - \bm{p}_h^n + \delta t \nabla \bar{u}_h \rangle = 0 & \forall \bm{w} \in \mathrm{CG}_p^2,
\end{align}
\end{subequations}
\endgroup
for midpoints $\bar{u}_h = (u_h^{n+1} + u_h^n)/2$, $\bar{\bm{p}}_h = (\bm{p}_h^{n+1} + \bm{p}_h^{n})/2$, and boundary penalty parameter $\kappa_p = 1$.
Rearranging these equations results in a linear system for the unknown quantities $[u_h^{n+1}, \bm{p}_h^{n+1}]$ in terms of the known quantities $[u_h^{n}, \bm{p}_h^{n}]$.
We consider a computational mesh with $n_x \times n_y$ quadrilateral elements, where $n_x = n_y = 3 \cdot 2^r$ for refinement $r \in \mathbb{N}$.

\subsection{Solver setup}
\label{sec:num-solver-setup}

Now we discuss solver details.
All numerical linear algebra, including solving generalized eigenvalue problems, is performed with SciPy \cite{Scipy-2020}.
When considering two-level solves of a linear system $A \bm{x} = \bm{b}$, we initialize the algorithm over ten randomly generated starting vectors $\bm{x}_0^j \approx \bm{x}$, $j = 1, \ldots, 10$, with the sequence of associated two-level iterates denoted by $\{ \bm{x}_k^j \}_{j = 0}^{k_{\max}}$.
For a fixed $j$, the two-level method is iterated until $k_{\max} := k_{\max}(j)$ iterations, which is the minimum of 20, and the smallest $k$ such that ${\Vert \bm{r}^j_k \Vert / \Vert \bm{r}^j_0 \Vert \leq 10^{-10}}$, for residual $\bm{r}_k = \bm{b} - A \bm{x}_k$. 
For each two-level method, we present the worst-case, geometrically-averaged, $\ell^2$ numerical convergence factors for error and residual, which we define as
\begin{subequations} \label{eq:prop-norm-num-def}
\begin{align}
    \Vert R_{\rm TG}^{\nu_1, \nu_2} \Vert_{\rm num} &:= 
    \max_{j} \left( \frac{ \Vert\bm{r}_{k_{\max}}^j \Vert}{ \Vert \bm{r}_{0}^j \Vert} \right)^{1/{k_{\max}}}
    =
    \max_{j} \left( \frac{
    \Vert (R_{\rm TG}^{\nu_1, \nu_2})^{k_{\max} }
    \bm{r}_{0}^j \Vert }
    {\Vert \bm{r}_{0}^j \Vert} \right)^{1/{k_{\max}}},
    \\
    \Vert E_{\rm TG}^{\nu_1, \nu_2} \Vert_{\rm num} &:= 
    \max_{j} \left( \frac{\Vert \bm{e}_{k_{\max}}^j \Vert}{\Vert \bm{e}_{0}^j \Vert} \right)^{1/{k_{\max}}}
    =
    \max_{j} \left( \frac{\Vert
    (E_{\rm TG}^{\nu_1, \nu_2})^{k_{\max}}
    \bm{e}_{0}^j \Vert}
    {\Vert \bm{e}_{0}^j \Vert} \right)^{1/{k_{\max}}}.
\end{align}
\end{subequations}
Here, $R_{\rm TG}^{\nu_1, \nu_2}$ is the two-level residual propagation matrix, and is similar to error propagation, $R_{\rm TG}^{\nu_1, \nu_2} := A E_{\rm TG}^{\nu_1, \nu_2} A^{-1}$. 
For each two-level solver, we also explicitly compute $\Vert E^{\nu_1, \nu_2}_{\rm TG} \Vert_{\cal N} 
= 
\Vert ({\cal D} V_r^{-1}) E_{\rm TG}^{ \nu_1, \nu_2 } ({\cal D} V_r^{-1})^{-1} \Vert$ (see \eqref{eq:calM-norm-def-mat}), with $E_{\rm TG}^{ \nu_1, \nu_2 }$ formed according to \eqref{eq:two-grid}, and with ${\cal D} = I$ for simplicity.
In all tests, the number of pre- and post-preconditioning steps is fixed at $\nu_1 = \nu_2 = 1$.

Our numerical results compare \eqref{eq:prop-norm-num-def} with $|1 - \lambda_{n_c+1}|^{\nu_1 + \nu_2}$, which is the the theoretically predicted value for $\Vert E^{\nu_1, \nu_2}_{\rm TG} \Vert_{\cal N}$, and for $\rho(E^{\nu_1, \nu_2}_{\rm TG} )$ according to \cref{thm:two-grid-conv,thm:real-valued}.
Note that as $k_{\max} \to \infty$, we have $\Vert R_{\rm TG}^{\nu_1, \nu_2} \Vert_{\rm num} \to |1 - \lambda_{n_c+1}|^{\nu_1 + \nu_2}$ and $\Vert E_{\rm TG}^{\nu_1, \nu_2} \Vert_{\rm num} \to |1 - \lambda_{n_c+1}|^{\nu_1 + \nu_2}$, because the quantities in \eqref{eq:prop-norm-num-def} limit to the spectral radii of the associated propagators as $k_{\max} \to \infty$, and note that the propagators have the same spectral radii since they are similar.

\begin{figure}[b!]
    \centering
    \includegraphics[width=0.475\linewidth]{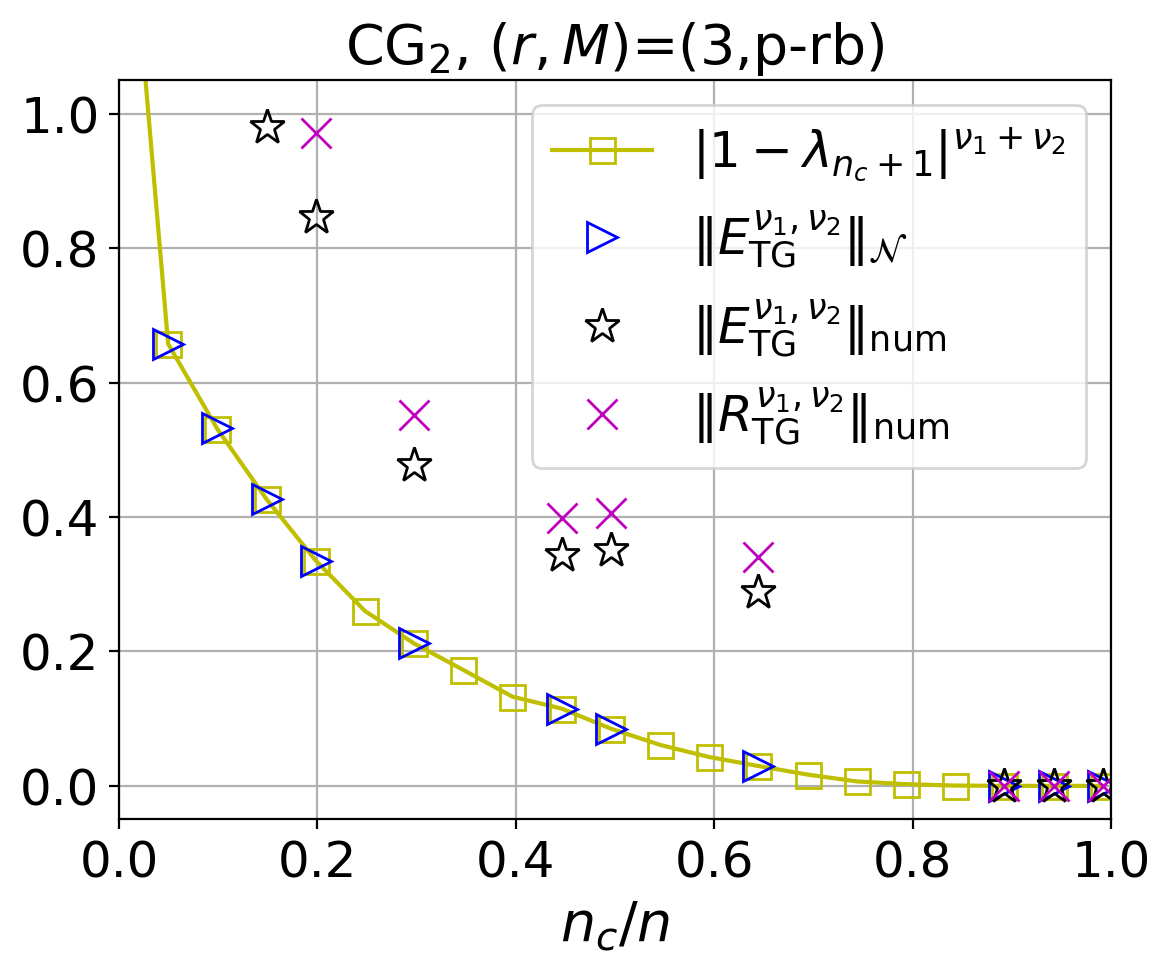}
    \hspace{1.5ex}
    \includegraphics[width=0.475\linewidth]{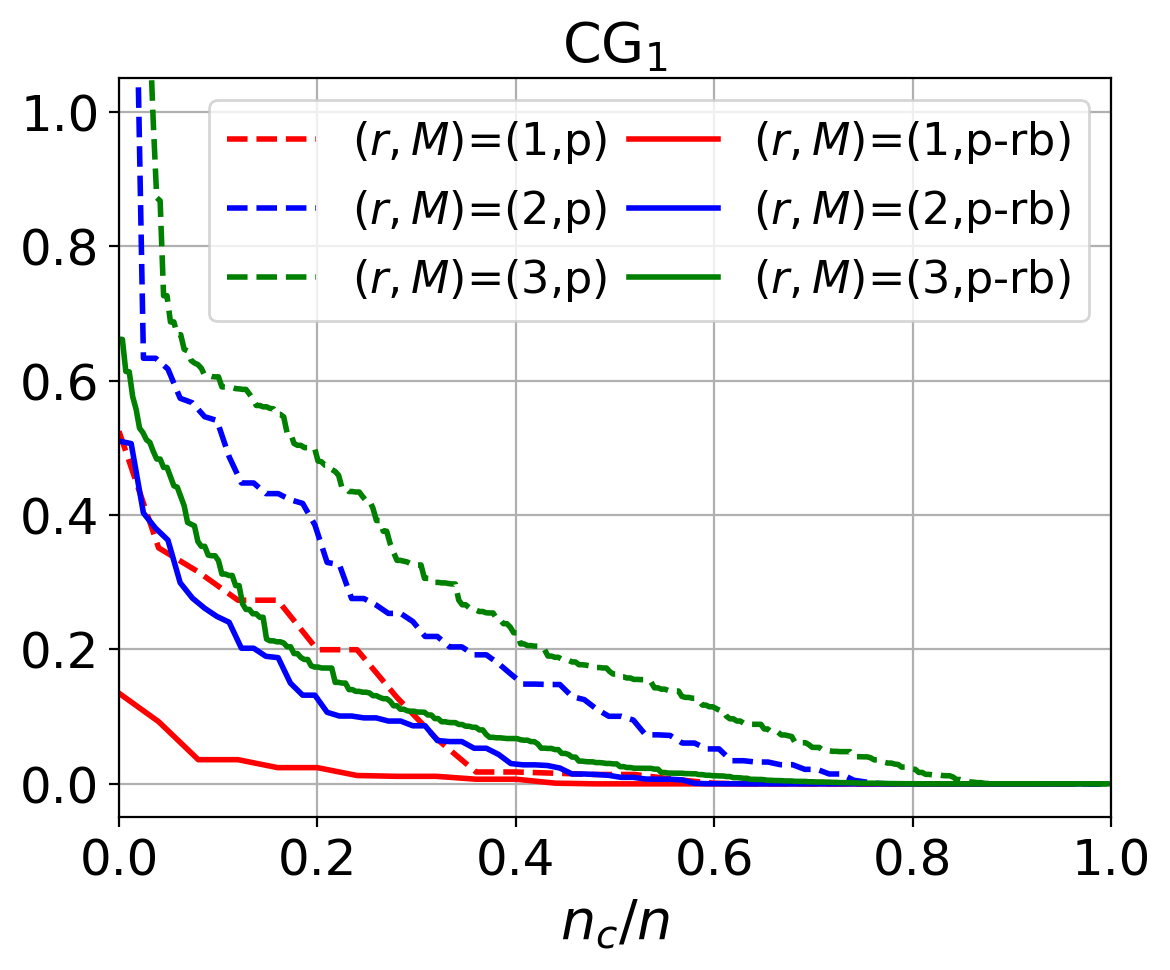}

    \vspace{0.5ex}
    
    \includegraphics[width=0.475\linewidth]{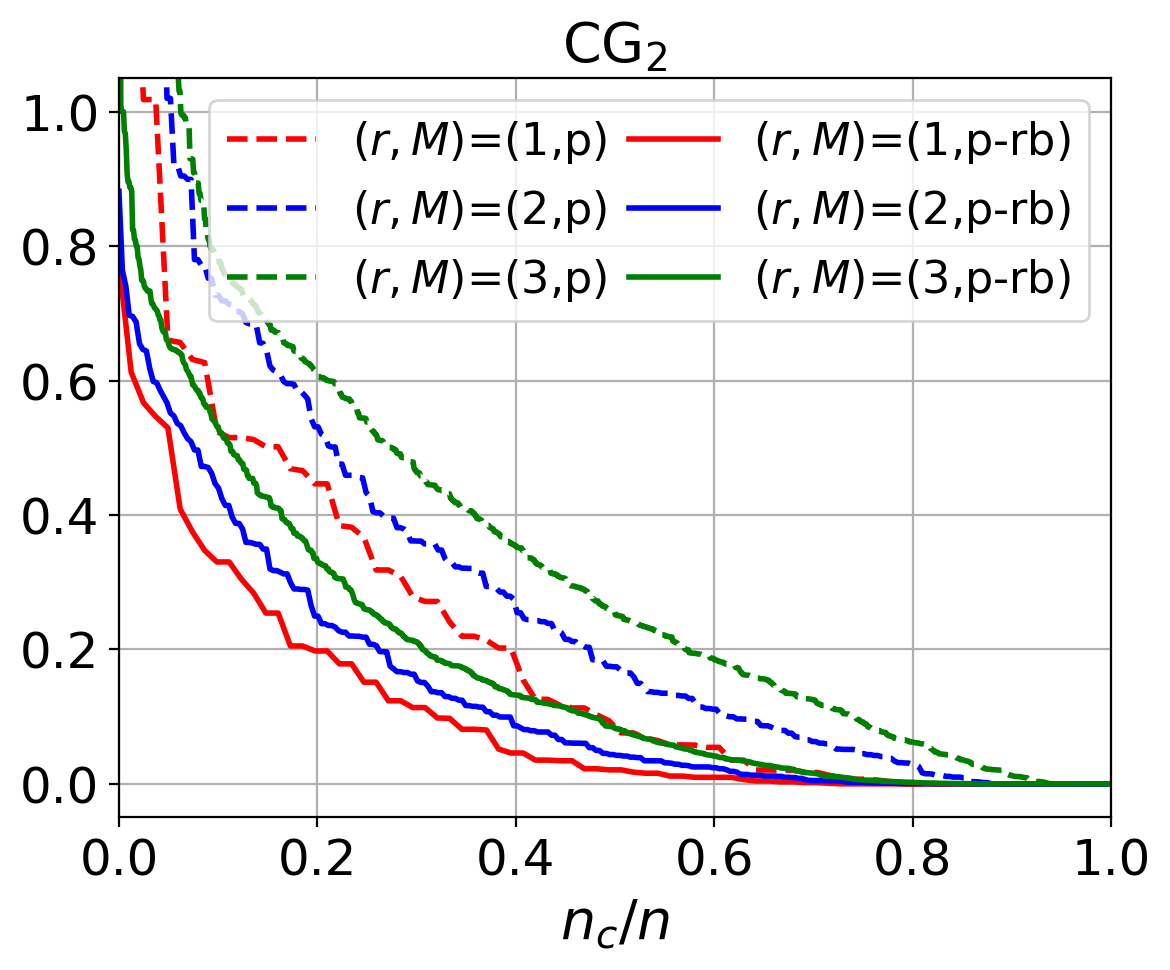}
    \hspace{1.5ex}
    \includegraphics[width=0.475\linewidth]{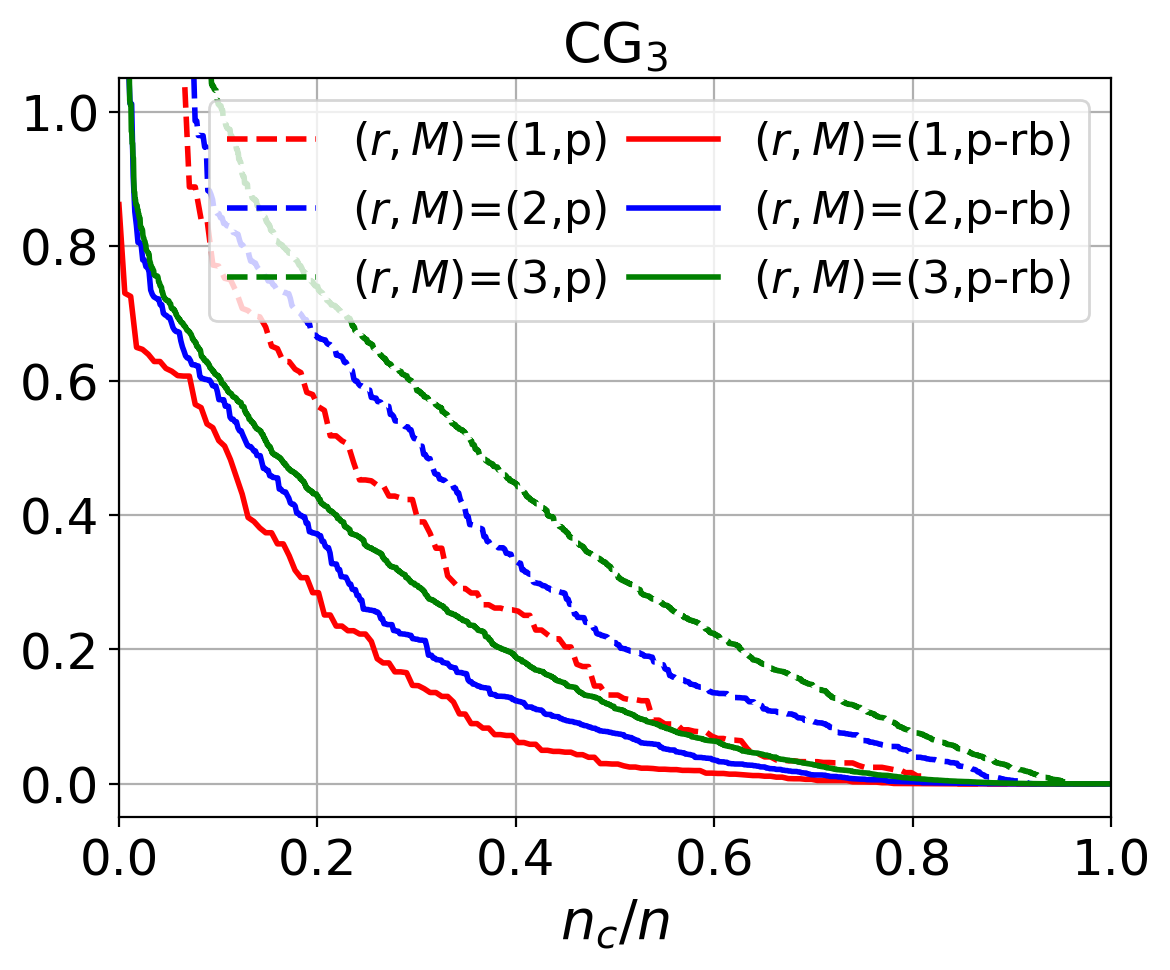}

    \caption{
    $\mathrm{CG}_p$ discretization \eqref{eq:adv-rea-CG-def} of the advection-reaction problem \eqref{eq:adv-rea} with refinement $r$ and fine-space preconditioner $M$.
    Top left: Numerical verification of theory for the specific example of CG$_2$, $r = 3$, and $M = $ point-wise red-black Jacobi.
    Remaining plots are $|1 - \lambda_{n_c+1}|^{\nu_1 + \nu_2}$ for polynomial degrees $p \in \{1, 2, 3\}$, as indicated in their titles. 
    \label{fig:advect-react-CG}
    }
\end{figure}

Now we discuss the fine-space preconditioner $M$.
For simplicity, we consider only Jacobi-based methods for $M$. 
We consider red-black Jacobi as well as standard Jacobi, with each method also being considered in both point-wise and block-wise fashion. 
Given a partitioning of DOFs into red ``$r$'' and black ``$b$'' points, a red-black Jacobi iteration first does a standard Jacobi update only on the red DOFs, and then uses this updated information at red points to do a Jacobi update only on the black DOFs. 
One can show, e.g., see \cite[Sec. 3.3]{southworth2019necessary} where red-black Jacobi corresponds to FC-relaxation, that the preconditioner matrix $M$ for red-black Jacobi is given by
\begin{align} \label{eq:M-def-rb-jac}
    M
    =
    {\cal P}
    \begin{bmatrix}
        D_{rr} & 0 \\
        A_{br} & D_{bb} 
    \end{bmatrix}
    {\cal P}^\top,
\end{align}
where $D$ is the diagonal of $A$, and ${\cal P} \colon \mathbb{R}^n \to \mathbb{R}^n$ is a permutation matrix mapping vectors from their underlying ordering into one where all red points are blocked together followed by all black points.
The matrix \eqref{eq:M-def-rb-jac} is clearly nonsymmetric so long as $\dim(r), \dim(b) > 0$.
For $\mathrm{DG}_p$ we also consider block Jacobi, where DOFs in an element are blocked together, such that each block has $(p+1)^2$ DOFs.
Block red-black Jacobi generalizes the point-wise version, where one instead partitions the blocked DOFs into red blocks and black blocks, and the red-only and black-only standard Jacobi steps are replaced with red-only and black-only block Jacobi steps; the form of $M$ is analogous to that in \eqref{eq:M-def-rb-jac}, except that the matrices are replaced with their block generalizations.
We partition DOFs into red and black points (or DG elements into red and black blocks in the block case) using the Ruge--St{\"u}ben ``CF'' splitting algorithm \cite{Ruge-Stuben-1987}, as implemented in PyAMG \cite{Bell-etal-2023}, where we assign red points as the resulting ``F-points,'' and black points as the ``C-points'' (the CF points from the Ruge--St{\"u}ben splitting should not be confused with the genuine coarse-fine splitting of DOFs induced by the smallest generalized eigenvectors of $(A, M)$). Hence, this preconditioner mimics FC-Jacobi, which we have used previously as a relaxation method when applying AMG to advection-related problems \cite{manteuffel2018nonsymmetric,Ali-etal-2024-optimalP-gen}.
In the Ruge--St{\"u}ben splitting, the classical strength measure is used, with a threshold of $\theta = 0.25$.  
In certain plots, lines are indicated with the pair $(r,M)$, for $r$ mesh refinement levels and fine-space preconditioner $M \in\{\mathrm{p}$,$\mathrm{b}$,$\mathrm{p}$-$\mathrm{rb}$,$\mathrm{b}$-$\mathrm{rb}\}$, where $\mathrm{p}$ and $\mathrm{b}$ denote point-wise or block Jacobi methods, respectively, -$\mathrm{rb}$ indicates a red-black splitting. For example, $\mathrm{p}\equiv$ point-wise Jacobi, and $\mathrm{p}$-$\mathrm{rb}\equiv$ point-wise red-black Jacobi.

\subsection{Advection-reaction equation}
\label{sec:num-adv-rea}

\begin{figure}[b!]
    \centering
    \includegraphics[width=0.475\linewidth]{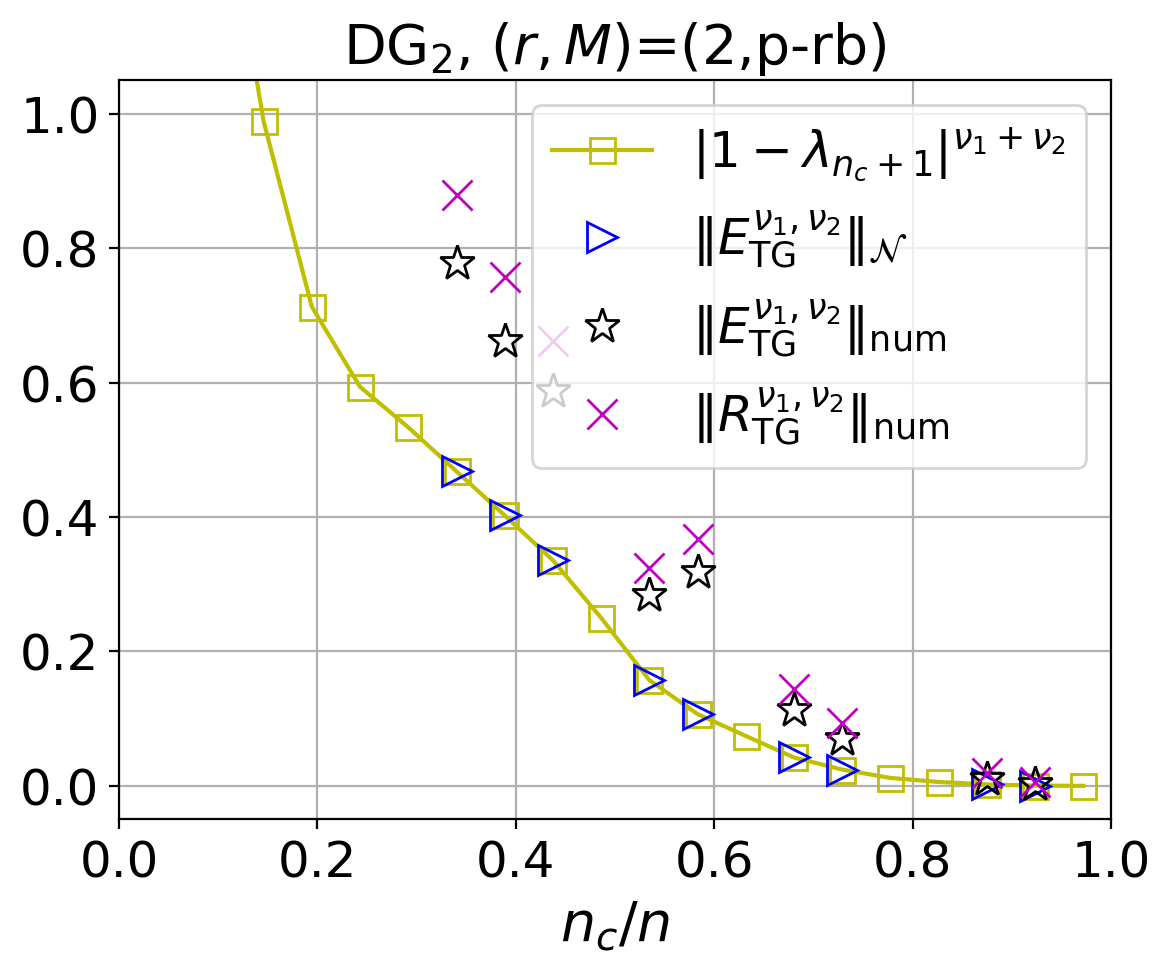}
    \hspace{1.5ex}
    \includegraphics[width=0.475\linewidth]{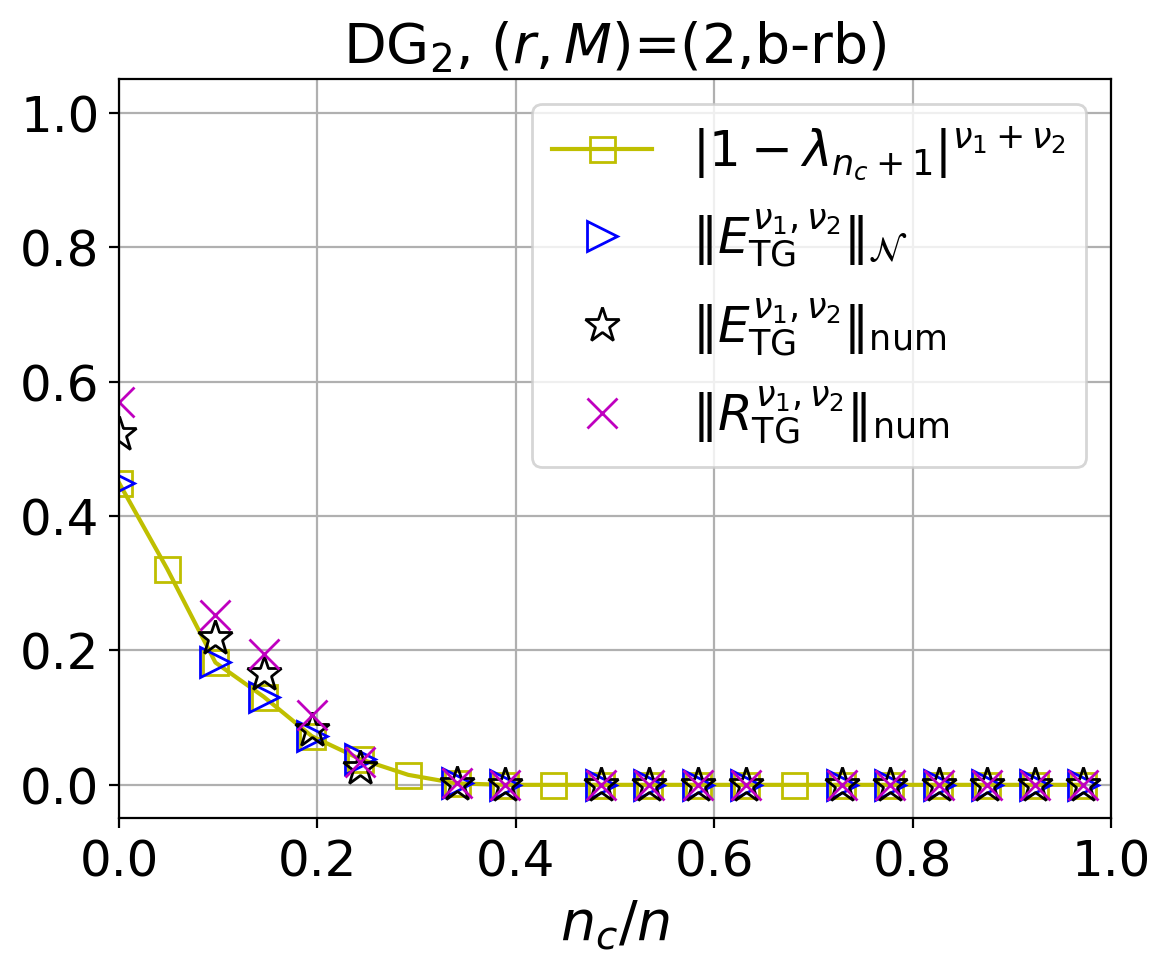}

    \vspace{0.5ex}

    \includegraphics[width=0.475\linewidth]{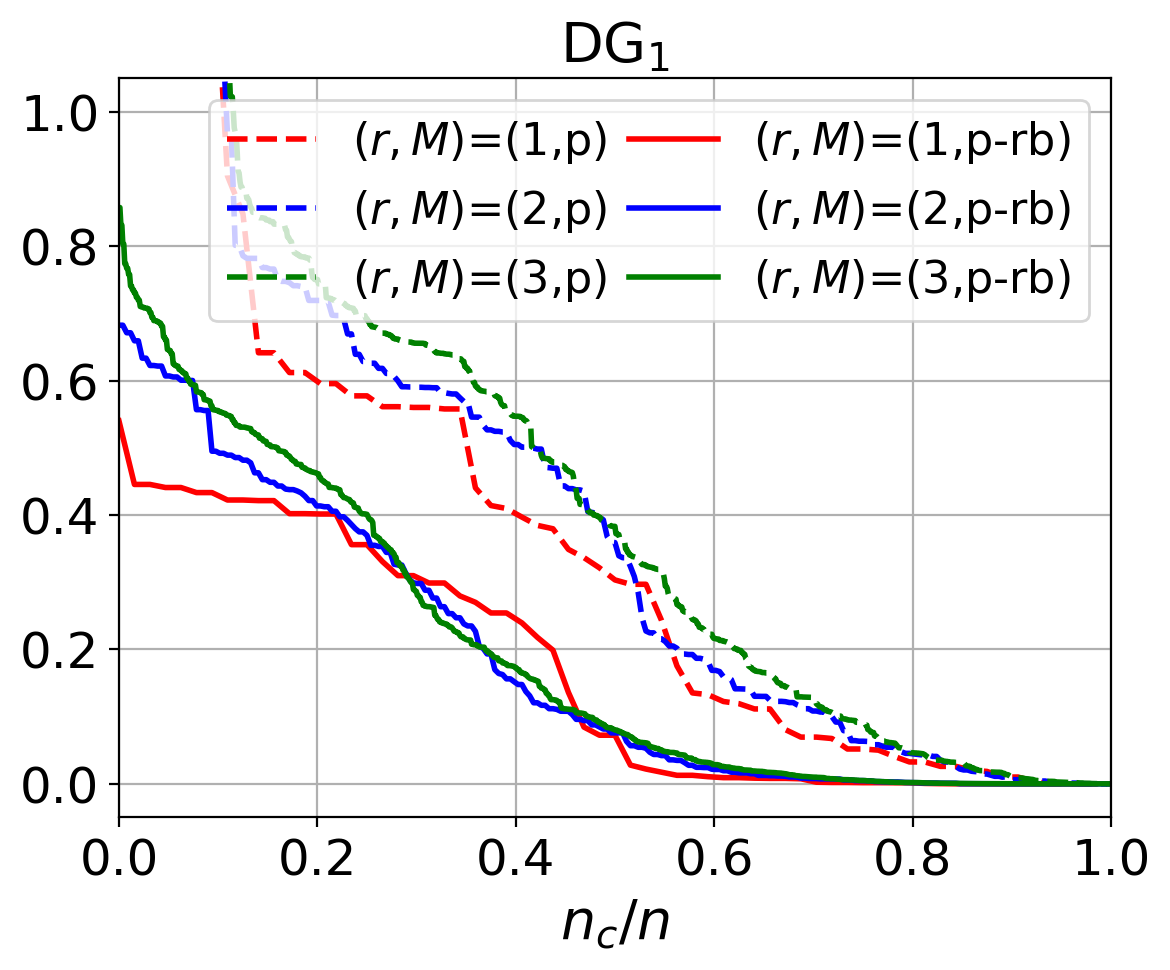}
    \hspace{1.5ex}
    \includegraphics[width=0.475\linewidth]{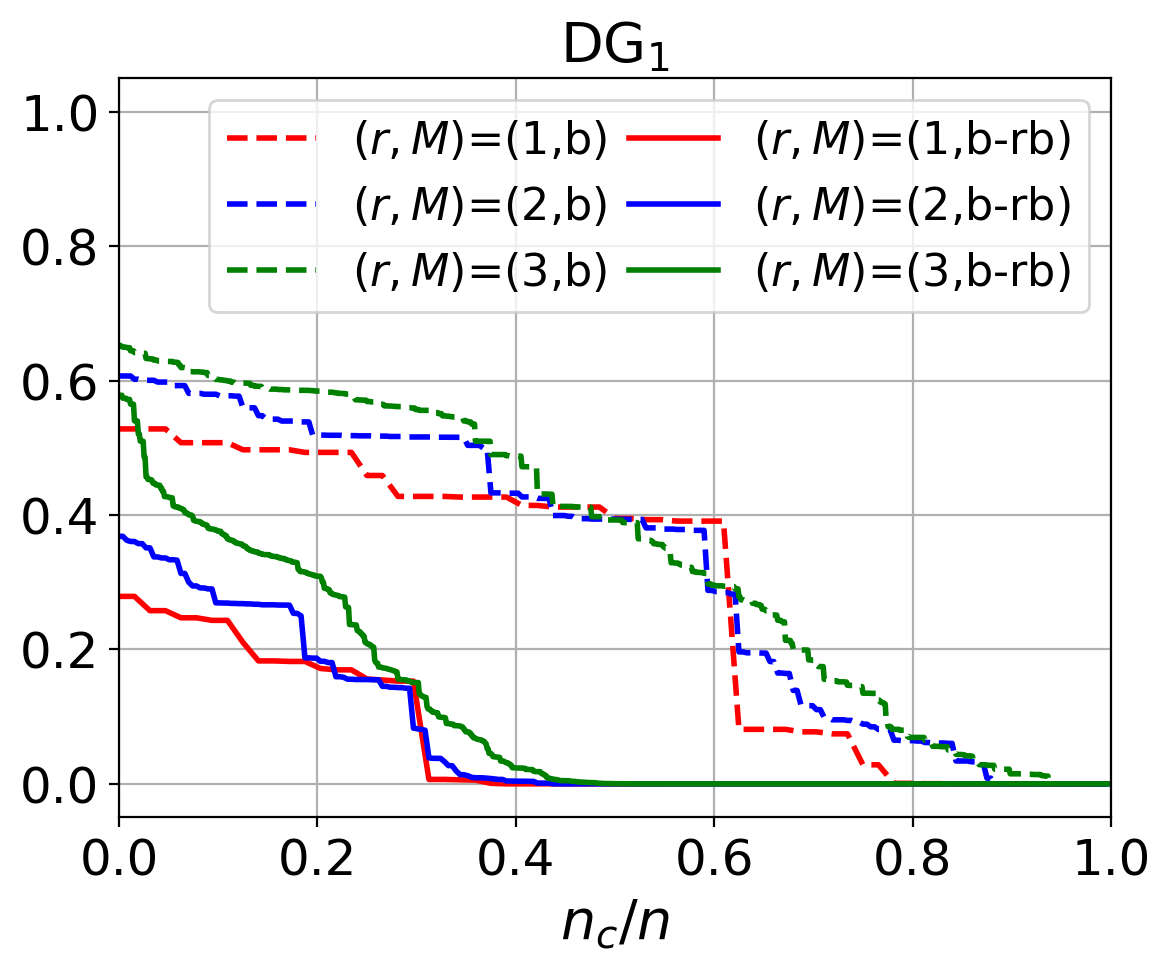}

    \vspace{0.5ex}
    
    \includegraphics[width=0.475\linewidth]{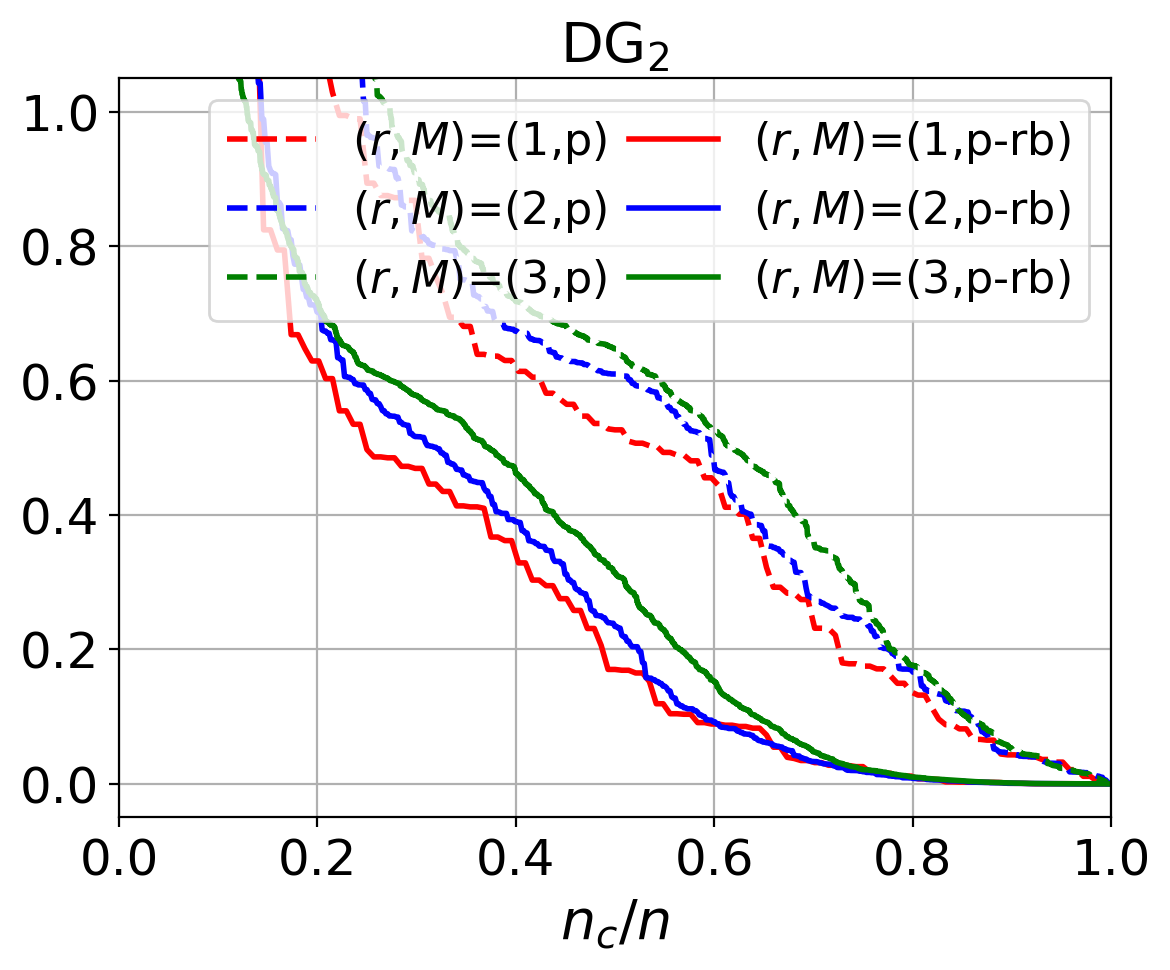}
    \hspace{1.5ex}
    \includegraphics[width=0.475\linewidth]{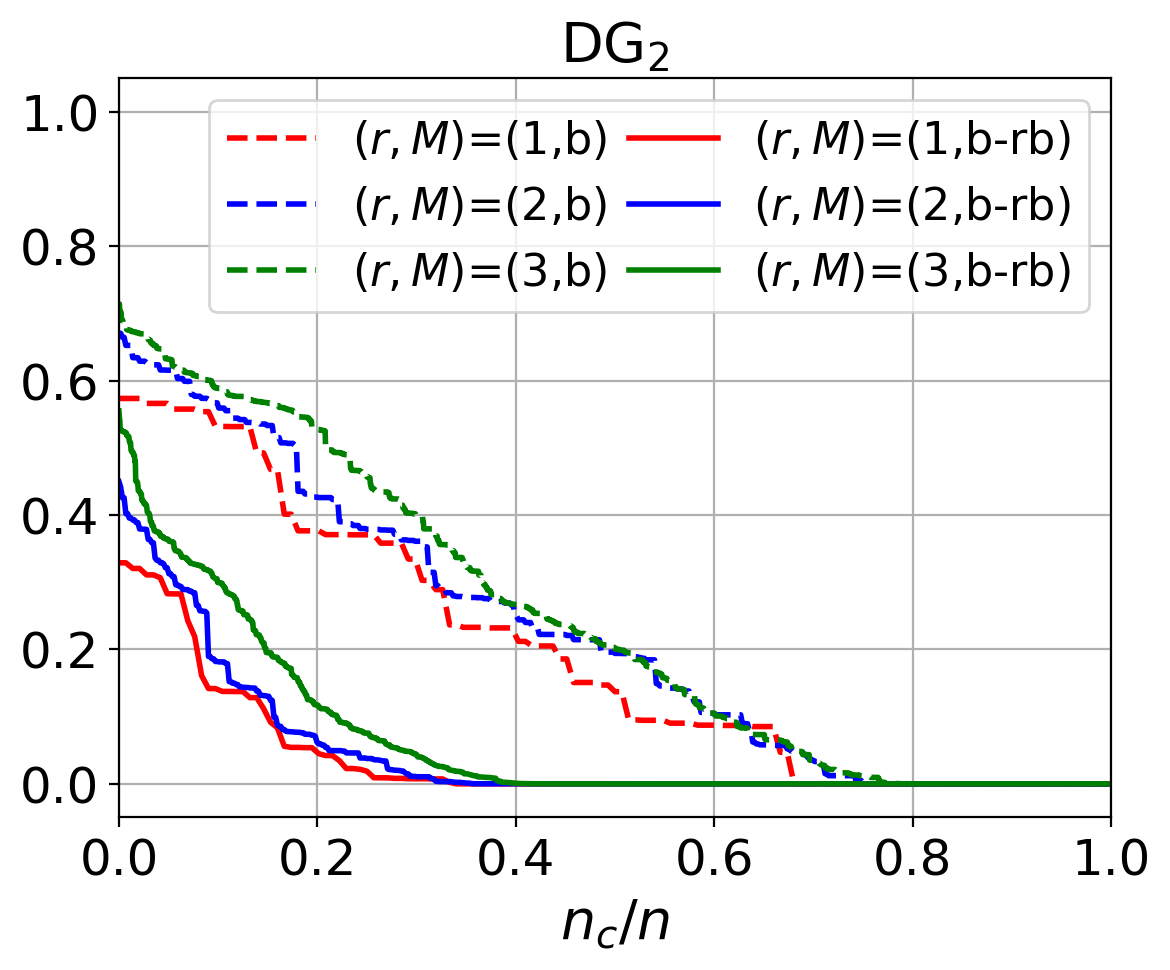}

    \caption{
    $\mathrm{DG}_p$ discretization \eqref{eq:adv-rea-DG-def} of the advection-reaction problem \eqref{eq:adv-rea} with refinement $r$ and fine-space preconditioner $M$.
    \textbf{Left column:} point-wise Jacobi.
    \textbf{Right column:} block Jacobi.
    Top row: Numerical verification of theory for DG$_2$ with refinement $r = 2$, and point-wise (left) and block (right) red-black Jacobi.
    Middle and bottom rows are plots of $|1 - \lambda_{n_c+1}|^{\nu_1 + \nu_2}$ for polynomial degrees $p$ indicated in the titles. 
    \label{fig:advect-react-DG}
    }
\end{figure}

Now we consider numerical results for the advection-reaction problem \eqref{eq:adv-rea}, beginning with the $\mathrm{CG}_p$ discretization \eqref{eq:adv-rea-CG-def} shown in \cref{fig:advect-react-CG}.
Considering the top left panel in \cref{fig:advect-react-CG}, we see exact agreement between numerically computed values of $\Vert E_{\rm TG}^{\nu_1, \nu_2} \Vert_{\cal N}$ and $|1 - \lambda_{n_c+1}|^{\nu_1 + \nu_2}$, as predicted by theory.
Observe that the numerically observed residual and error reduction are somewhat larger than the value of $|1 - \lambda_{n_c+1}|^{\nu_1 + \nu_2}$, which they ultimately limit to as $k_{\max} \to \infty$.
Considering the other plots in \cref{fig:advect-react-CG}, we see that red-black Jacobi always results in faster two-level convergence than standard Jacobi. It is also interesting to observe that two-level convergence, for a fixed $n_c/n$, seems to be roughly scalable with respect to mesh resolution $r$ and polynomial order $p$, particularly for larger $p$. 
Finally, notice in many cases that the two-level method is not convergent when the coarsening is extremely aggressive, $n_c/n \approx 0$, but that for moderate coarsening rates two-level convergence is always restored.

Next we consider the advection-reaction problem with the DG$_{p}$ discretization \eqref{eq:adv-rea-DG-def} in \cref{fig:advect-react-DG}.
Many of the trends seen in the $\mathrm{CG}_p$ case (see \cref{fig:advect-react-CG}) carry over to the $\mathrm{DG}_p$ case. Considering the top right panel, it is interesting that there is much closer agreement between the numerically observed convergence factors and their asymptotic limit, which indicates that the pre-asymptotic convergence phase in this case is relatively shorter than for the other problems.
Considering the other plots, there is, in most cases, a stark contrast between block and point-wise Jacobi around $n_c/n \approx 0$, where the block methods converge fairly quickly, while the point-wise-based methods diverge in most cases.
However, standard block Jacobi is not always stronger than standard point-wise Jacobi, e.g., for DG$_1$ around $n_c/n \approx 0.6$, where two-level convergence for point-wise is faster. 
As in \cref{fig:advect-react-CG}, dramatic improvement occurs from using red-black-based Jacobi compared to standard Jacobi.

\subsection{Mixed wave equation}
\label{sec:num-wave}

\begin{figure}[b!]
    \centering
    \includegraphics[width=0.475\textwidth]{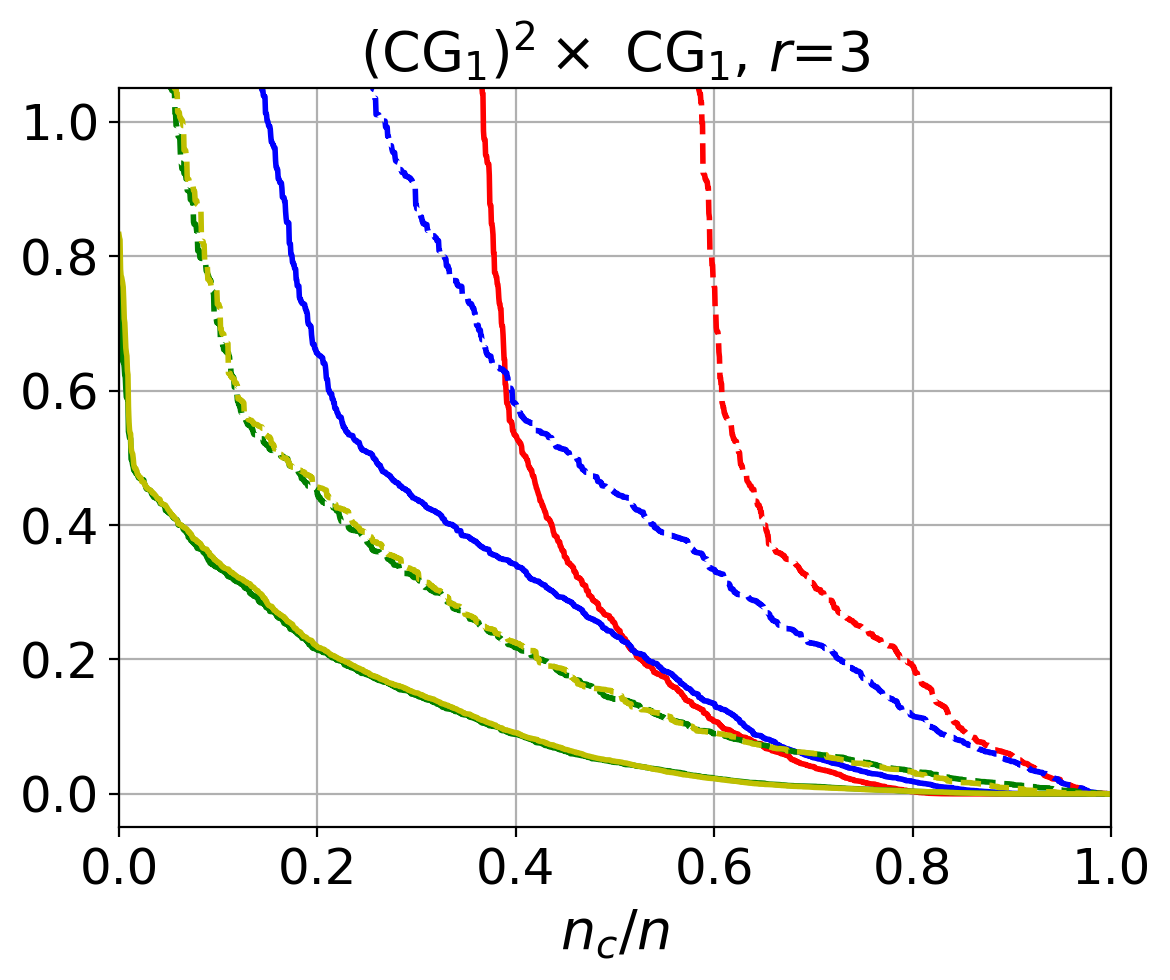}
    \hspace{1.5ex}
    \includegraphics[width=0.475\textwidth]{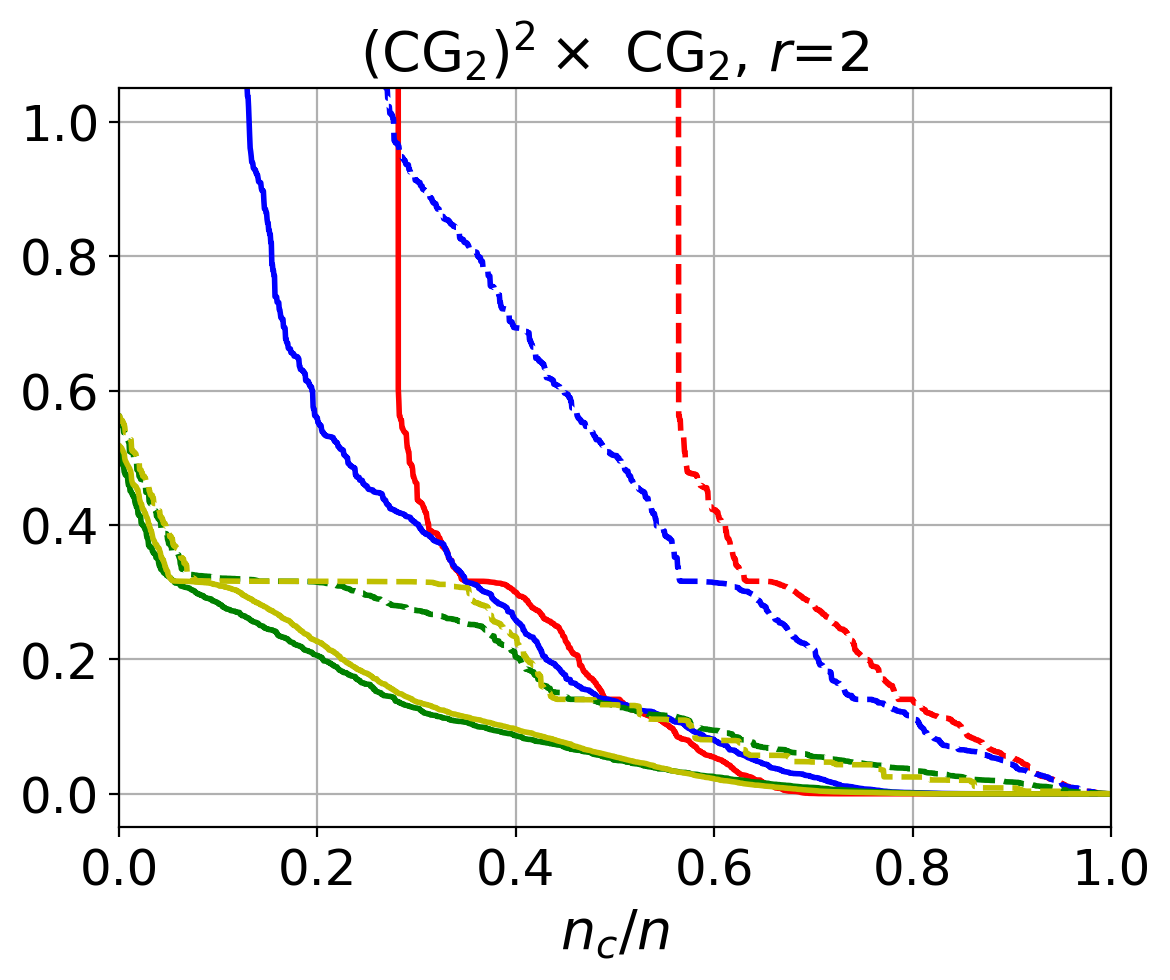}

    \includegraphics[width=1\textwidth]{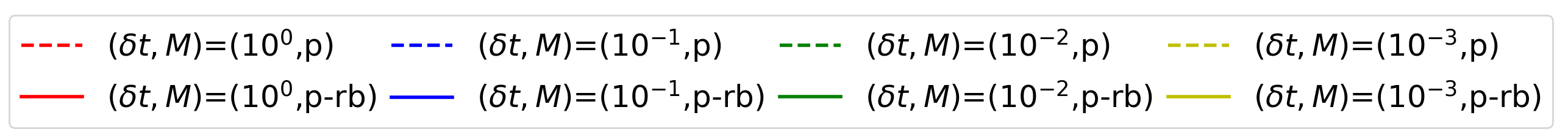}
    
    \caption{
    $(\mathrm{CG}_p)^2 \times (\mathrm{CG}_p)$ discretization \eqref{eq:wave-disc} of mixed wave equation with \eqref{eq:wave} with refinement $r$, time-step size $\delta t$ and fine-space preconditioner $M$ (where ``p'' $\equiv$ point-wise Jacobi, and ``p-rb'' $\equiv$ point-wise red-black Jacobi).
    Plots show $|1 - \lambda_{n_c+1}|^{\nu_1 + \nu_2}$ for polynomial degrees $p$ and refinements $r$ indicated in the titles. 
    All tests have $n = 1875$ DOFs.
    \label{fig:wave}
    }
\end{figure}

Now we consider the mixed wave equation \eqref{eq:wave}.
In particular, we consider the linear system resulting from the discretization \eqref{eq:wave-disc} for a fixed time-step size $ \delta t \in \{ 1, 10^{-1}, 10^{-2}, 10^{-3} \}$.
For each $\delta t$ we examine the optimal two-level convergence factor ${|1 - \lambda_{n_c+1}|^{\nu_1 + \nu_2}}$ to understand to what extent the resulting linear system is solvable (or not) with a two-level method given  our choice of $M$.
For preconditioners $M$ we consider both point-wise Jacobi and point-wise red-black Jacobi. 
Results are shown in \cref{fig:wave} for two different polynomial degrees $p$.
First, observe that red-black is significantly stronger than standard Jacobi almost everywhere, just as was the case for the advection-reaction problem (see \cref{fig:advect-react-CG,fig:advect-react-DG}).
Secondly, observe that the system gets progressively harder to solve for larger $\delta t$, corresponding to increasing stiffness. 
Finally, it is rather interesting to consider the implication of ${|1 - \lambda_{n_c+1}|^{\nu_1 + \nu_2}}$ being larger than unity for quite mild coarsening rates $n_c/n$, particularly for larger values of $\delta t$. That is, recall from \cref{cor:nec-and-suf} that for fixed $n_c$, the condition $|1 - \lambda_{n_c+1}| < 1$ is \emph{necessary} for there to even exist a convergent two-level method (given $M$). 
For example, these plots tell us that when $\delta t = 1$ and using a Jacobi preconditioner, one cannot coarsen by more than a factor of two ($n_c / n < 0.5$) and have a convergent method, \emph{regardless} of the interpolation and restriction used. We also tested a block Jacobi preconditioner, where we group nodal CG DOFs associated with the scalar and vector variables into $3\times 3$ blocks. This block-based method offered marginal to no improvement over point-wise Jacobi methods (results are not shown, both for brevity and unimpressive performance).

\section{Conclusions}
\label{sec:con}

We generalize the optimal interpolation framework originally proposed by Brannick et al. \cite{Brannick-etal-2018-optimalP} for HPD linear systems, and then recently extended to non-HPD systems by Ali et al. \cite{Ali-etal-2024-optimalP-gen}.
Specifically, we consider transfer operators constructed from the smallest (in a certain sense) $n_c < n$ generalized left (restriction) and right (interpolation) eigenvectors of the matrix pencil $(A, M)$, with fine-space preconditioner $M \approx A$. 
Complex-valued realizations of these transfer operators $\{R_\#,P_\#\}$ were already proposed in \cite{Ali-etal-2024-optimalP-gen}, but therein genuine optimality was not shown, and only asymptotic convergence of the associated two-level method was proven.
Here we develop tight, norm-based convergence bounds, and show genuine optimality of these transfer operators with respect to norm-based convergence.
We also describe how to construct real-valued optimal transfer operators that yield identical convergence (and corresponding optimality) when the pencil $(A, M)$ has real-valued, non-SPD matrices, since those that naturally arise from the generalized eigenvalue problem are almost certainly complex valued in this case.

The optimal transfer operators considered herein are not practical because they are dense, and their construction requires solving global-sized generalized eigenvalue problems.
In future work we will pursue practical two-level methods through local approximations to these transfer operators.

\bibliographystyle{siamplain}
\bibliography{Optimal-P-SIMAX/optimal-P-refs}
\end{document}